\documentclass[12pt,amstex]{amsart}

\usepackage{mathptmx}
\usepackage{mathrsfs}

\usepackage{verbatim}
\usepackage{url}
\usepackage[all]{xy}
\usepackage{color}

\usepackage[colorlinks=true,citecolor=blue]{hyperref}

\usepackage{stmaryrd}
\usepackage{epsfig}
\usepackage{amsmath}
\usepackage{amssymb}

\usepackage{amscd}
\usepackage{graphicx}
\usepackage{pstricks}
\usepackage{tikz}

\theoremstyle{plain}
\newtheorem{thm}{Theorem}[section]
\newtheorem{lem}[thm]{Lemma}

\newtheorem{prop}[thm]{Proposition}

\newtheorem{cor}[thm]{Corollary}
\theoremstyle{definition}
\newtheorem{de}[thm]{Definition}
\newtheorem{exam}[thm]{Example}
\theoremstyle{remark}
\newtheorem{rem}[thm]{Remark}
\newtheorem{conj}{\bf Conjecture}

\numberwithin{equation}{section}

\def \N {\mathbb{N}}

\def \Z {\mathbb{Z}}

\def \O {\mathcal{O}}

\def \A {\mathcal A}

\def \X {\mathcal{X}}

\def \RP {{\bf RP}}

\def \id {{\rm id}}

\def \a {\alpha }
\def \b {\beta}
\def \ep {\epsilon}
\def \d {\delta}
\def \D {\Delta}
\def \c {\circ}
\def \w {\omega}

\def \ra {\rightarrow}

\begin{document}
\title{Multiple  recurrence without commutativity}

\author{Wen Huang}
\author{Song Shao}
\author{Xiangdong Ye}

\address{School of Mathematical Sciences, University of Science and Technology of China, Hefei, Anhui, 230026, P.R. China}

\email{wenh@mail.ustc.edu.cn}
\email{songshao@ustc.edu.cn}
\email{yexd@ustc.edu.cn}

\subjclass[2010]{Primary: 37B05; 54H20}
\keywords{Saturation theorems, integral polynomials, almost one to one extensions, pro-nilsystems}

\thanks{This research is supported by National Natural Science Foundation of China (12371196, 12031019, 12090012).}


\begin{abstract}
We study multiple  recurrence without commutativity in this paper. We show that for any two homeomorphisms $T,S: X\rightarrow X$ with $(X,T)$ and $(X,S)$ being minimal, there is a residual subset $X_0$ of $X$ such that for any $x\in X_0$ and any nonlinear integral polynomials
$p_1,\ldots, p_d$ vanishing at $0$, there is some subsequence $\{n_i\}$ of $\Z$ with $n_i\to \infty$ satisfying
$$  S^{n_i}x\to x,\ T^{p_1(n_i)}x\to x, \ldots,\ T^{p_d(n_i)}x\to x,\ i\to\infty.$$
\end{abstract}

\maketitle




\section{Introduction}

In this section we will provide the motivation of the research, state the main results and some open questions. In the paper an {\em integral polynomial} is a polynomial
taking integer values at the
integers. The polynomials $p(n)$ and $q(n)$ are {\em essentially distinct} if $p(n)-q(n)$ is not a constant function. By a {\em topological dynamical system} (t.d.s. for short) we mean a pair $(X,T)$, where $X$ is a compact metric space (with a metric $\rho$) and $T:X \to X$ is a homeomorphism.

\subsection{Motivation}\
\medskip

Recurrence is one of the most important properties for dynamical systems, since it has important applications in dynamical systems and other fields of mathematics (for example see \cite{F}). The well known Birkhoff theorem says that for any t.d.s. $(X,T)$ there exists some recurrent point, i.e. there are some point $x\in X$ and a subsequence $\{n_i\}$ of $\Z$ such that $n_i\to\infty$ and $T^{n_i}x\to x, i\to\infty$. The multiple Birkhoff recurrence theorem says that if $(X,T)$ is a t.d.s.
and $d\in \N$, then there exist some $x\in X$ and a subsequence $\{n_i\}$ of $\Z$ with $n_i\to \infty$ and $$T^{n_i}x\to x,T^{2n_i}x\to x,\ldots, T^{dn_i}x\to x, i\to\infty.$$
The multiple Birkhoff recurrence theorem can be deduced from the multiple recurrence theorem
of Furstenberg \cite{F77} which was proved by using deep ergodic tools. It is Furstenberg and Weiss \cite{FW} who
presented a topologically dynamical proof  of the theorem.
Also in \cite{FW}, the multiple Birkhoff recurrence theorem is generalized to commuting transformations: if $T_1,\ldots, T_d: X\rightarrow X$ are commuting homeomorphisms, then there is are some $x\in X$ and  a subsequence $\{n_i\}_{i=1}^\infty$ of $\Z$ with $n_i\to \infty$  such that
$$T_1^{n_i}x\to x, T_2^{n_i}x\to x, \ldots, T_d^{n_i}x\to x,\ i\to\infty.$$
Note that
there are counterexamples when $\langle T_1,\ldots,T_d\rangle$ (the group generated by $T_1,\ldots,T_d$) is solvable \cite{BL02, F}.

Leibman showed the following polynomial multiple Birkhoff recurrence theorem \cite{Lei}.
Let $X$ be a compact metric space, let $\Gamma$ be a nilpotent
group of its homeomorphisms, let $T_1,\ldots ,T_d\in \Gamma$, let $k\in \N$ and let $p_{i,j}$ be integral polynomials with $p_{i,j}(0)=0$, $i=1,2,\ldots,k; j=1,2,\ldots, d$. Then there exist $x\in X$ and
$\{n_m\}_{m=1}^\infty$ such that for each $i=1,\ldots,k$,
\begin{equation*}
  T_1^{p_{i,1}(n_m)}\cdots T_d^{p_{i,d}(n_m)}x\rightarrow x, \ m\to \infty.
\end{equation*}
The special case of this theorem corresponding to commutative $\Gamma$ is
proved by Bergelson and Leibman in \cite{BL96}.
Note that when $(X,\Gamma)$ is minimal, the set of multiple recurrence points is a dense $G_\d$ set of $X$ \cite{Lei}.

\medskip

Recently, Host and Frantzikinakis studied multiple recurrence and convergence without commutativity in ergodic theory \cite{FranHost21}. They proved that
if $T$ and $S$ are ergodic measure preserving transformations on a probability space $(X,\X,\mu)$ and $T$ has zero entropy, then
for any integral polynomial $p$ with $\deg {p}\ge 2$, one has
$$\lim_{N\to\infty} \frac{1}{N}\sum_{n=0}^{N-1}f(T^{n}x)g(S^{p(n)}x),$$
exists in $L^2(X,\mu)$.

Note that there are examples such that if $\deg p, \deg {q}\ge 5$ and both $T$ and $S$ have zero entropy,
$\lim_{N\to\infty} \frac{1}{N}\sum_{n=0}^{N-1}f(T^{p(n)}x)g(S^{q(n)}x)$ does not exist, see \cite{HSY22-3}. And for $\deg p=\deg q=1$, see \cite{Tim}, \cite{HSY2024} and \cite{R}. Note that a very recent result by Kosloff and Sanadhya \cite{Kosloff} complete the picture by giving
examples for the rest cases, namely $\deg p, \deg {q}\ge 2.$


Huang, Shao and Ye \cite{HSY-new} extended the result of Host and Frantzikinakis by showing that
for all $c_i\in \Z\setminus \{0\}$, all integral polynomials $p_j$ with $\deg {p_j}\ge 2$, and for all $f_i, g_j\in L^\infty(X,\mu)$, $1\le i\le m$ and $1\le j\le d$,
$$\lim_{N\to\infty} \frac{1}{N}\sum_{n=0}^{N-1}f_1(T^{c_1n}x)\cdots f_m(T^{c_mn}x)\cdot g_1(S^{p_1(n)}x)\cdots g_d(S^{p_d(n)}x),$$
exists in $L^2(X,\mu)$.

\medskip
Thus, a natural question arises: what happens in the realm of topological dynamics? This paper will focus on addressing this question.

\medskip
In the study of the measure theoretical case, the disjointness of a $K$-system with any zero entropy ergodic system plays an essential role. Here
in the topological setting we will use the following disjoint theorem due to Furstenberg in \cite{F67}: If $(X,T)$ is a totally transitive t.d.s. and the set
of periodic points is dense in $X$, then $(X,T)$ is disjoint from any minimal t.d.s.




\subsection{Main results}\
\medskip

Surprisedly, in the topological setting we do not need to assume one the system has zero entropy. The deep reason can be explained as follows:
the Bernoulli shift is disjoint with any minimal system, and on the contrary in the ergodic case, no non-trivial egodic system is disjoint from any ergodic system.
\begin{thm}\label{thm-main1-1}
Let $(X,T)$ be a minimal t.d.s. and $p_1,\ldots, p_d$ be a family of {\bf nonlinear integral polynomials} vanishing at $0$.
Then there is a residual subset $X_0$ of $X$ such that for each $x\in X_0$ and any minimal point $z\in Z$, where $(Z,S)$ is any t.d.s., there is some sequence $n_i\to\infty, i\to\infty$ such that
\begin{equation}\label{int1-1}
  S^{n_i}z\to z,\ T^{p_1(n_i)}x\to x, \ldots,\ T^{p_d(n_i)}x\to\infty,\ i\to\infty.
\end{equation}
\end{thm}

The notion of product recurrence was introduced by Furstenberg in \cite{F}. Let $(X,T)$ be a t.d.s. A point $x\in X$ is said to be
{\em product recurrent} if given any recurrent point $y$ in any t.d.s. $(Y,S)$, $(x,y)$ is recurrent in the product system
$(X\times Y, T\times S)$. By associating product recurrence with a combinatorial property on the sets of return times (i.e. $x$ is
product recurrent if and only if it is $IP^*$ recurrent), Furstenberg  proved that product recurrence is equivalent to
distality \cite[Theorem 9.11]{F}. In \cite{AF}, Auslander and Furstenberg posed a question: if $(x,
y)$ is recurrent for all minimal points $y$, is $x$ necessarily a distal point? This question is answered in the negative in
\cite{HO}. Such $x$ is called a {\em weakly product recurrent} point there.

\medskip
We remark that Theorem \ref{thm-main1-1} is equivalent to
\begin{thm}\label{thm-main1-2}
Let $(X,T)$ be a minimal t.d.s. and $\A=\{p_1,\ldots,p_d\}$  be a family of nonlinear integral polynomials
vanishing at $0$.
Then there is a residual subset $X_0$ of $X$ such that for each $x\in X_0$, $\w_x^\A\in (X^d)^\Z$ is a weak product recurrent point, where
$$\w_x^\A=\big ((T^{p_1(n)}x, T^{p_2(n)}x,\ldots, T^{p_d(n)}x) \big)_{n\in \Z}  \in (X^d)^{{\Z}}$$
and the t.d.s. is $(\overline{\O}(\w_x^\A,\sigma), \sigma)$ with the shift $\sigma: (X^d)^{{\Z}}\rightarrow (X^d)^{{\Z}}$.
\end{thm}

By taking $(Z,S)=(X,T)$ and $z=x$ in Theorem \ref{thm-main1-1} we obtain

\begin{cor}
Let  $(X,T)$ and $(X,S)$ be minimal. Then there is a residual subset $X_0$ of $X$ such that for any $x\in X_0$ and any nonlinear integral polynomials $p_1,\ldots, p_d$ vanishing at $0$, 
there is some sequence $n_i\to\infty, i\to \infty$ such that
\begin{equation}\label{int1-2}
  S^{n_i}x\to x,\ T^{p_1(n_i)}x\to x, \ldots,\ T^{p_d(n_i)}x\to x,\ i\to\infty.
\end{equation}
\end{cor}

\medskip

It is clear that (\ref{int1-2}) is  equivalent to the following: for any nonempty open set $U$ of $X$, there is $n\in \N$ with
$$U\cap S^{-n}U\cap T^{-p_1(n)}U\cap \cdots \cap T^{-p_d(n)}U\not=\emptyset.$$
Thus, there is a residual subset $X_0$ is equivalent to there is a dense subset $X_0$ in the statement of the above corollary.

\medskip

Let $(X,T)$ be a minimal system and $k\in \N$. We say that $(X,T)$ is {\it multiply $k$-minimal}, if there is a dense $G_\delta$ set $\Omega$ of $X$ such that
for each $x\in \Omega$, $(x,\ldots,x)$ is minimal under $T\times T^2\times \cdots\times T^k$.
We have the following remarks:

\begin{enumerate}

\item It is clear that if $(X,T)$ is distal, then it is multiply $k$-minimal for any $k\in\N$. In fact, for distal systems $(X,T)$, for each $x\in \Omega$ and $k\in \N$, $(x,\ldots,x)$ is minimal under $T\times T^2\times \cdots\times T^k$ \cite{F81}.

\item There are minimal weakly mixing systems \cite{FKS73} and even some minimal open proximal extensions of equicontinuous systems \cite{LOU}, which are multiply $k$-minimal for some $k\ge 2$.

\item Assume that $\pi:(X,T)\to (Y,S)$ is a factor map between two minimal systems and $(Y,S)$ is multiply $k$-minimal for some $k\in\N$. If $\pi$ is distal then one can show that $(X,T)$ is also multiply $k$-minimal.

\item If there exists a multiply minimal point in any t.d.s. was a question asked by Furstenberg in \cite{F81}, and was systematically studied by Huang, Shao and Ye in \cite{HSY-21}. Particularly, it was shown that there is a minimal weakly mixing t.d.s. which has no multiply minimal point \cite{HSY-21}. It is an open question if such an example exists for PI-systems.
\end{enumerate}

\medskip

By taking $X_0$ to be the intersection of the residual set in Theorem \ref{thm-main1-2} and the residual set in the definition of multiply $k$-minimality we get

\begin{cor}
Let $(X,T)$, $(X,S)$ be minimal, and let $(X,S)$ be multiply $k$-minimal for some $k\in\N$. Then there is a residual subset $X_0$ of $X$ such that for any $x\in X_0$ and any finite nonlinear integer polynomials
$p_1,\ldots, p_d$ with $p_1(0)=\cdots =p_d(0)=0$, there is some sequence $n_i\to \infty, i\to\infty$ such that
\begin{equation}\label{}
  S^{n_i}x\to x,\  \ldots, S^{kn_i}x\to x, \ T^{p_1(n_i)}x\to x, \ldots,\ T^{p_d(n_i)}x\to x,\ i\to\infty.
\end{equation}
\end{cor}

We remark that the role of multiple $k$-minimality is the replacement of zero entropy in the ergodic case, as when two topological systems are disjoint, one of them should be minimal.

\subsection{More about main results}\
\medskip

A subset $S$ of $\Z$ is {\it syndetic} if it has a bounded gap,
i.e. there is $N\in \N$ such that $\{i,i+1,\cdots,i+N\} \cap S \neq
\emptyset$ for every $i \in {\Z}$. $S$ is {\it thick} if it
contains arbitrarily long runs of integers. A subset $S$ of $\Z$ is {\em piecewise syndetic} if it is an
intersection of a syndetic set with a thick set.

In fact, we will show a more general result than Theorem \ref{thm-main1-1}:

\begin{thm}\label{thm-main1}
Let $(X,T)$ be a minimal t.d.s. and $\A=\{p_1,\ldots, p_d\}$ be a family of nonlinear integer polynomials with $p_1(0)=\cdots =p_d(0)=0$. Then there is a residual subset $X_\bullet$ of $X$ such that for each $x_0\in X_\bullet$, any minimal point $z\in Z$, where $(Z,S)$ is any t.d.s.,
and any each $\ep>0$,
\begin{equation*}\label{}
 N_{z,x_0,\ep}=\{n\in \Z: \rho_Z(S^n z,z)<\ep \ \text{and}\ \rho_{X}(T^{p_1(n)}x_0, x_0)<\ep,\ \ldots,\ \rho_{X}(T^{p_d(n)}x_0, x_0)<\ep \}
\end{equation*}
is a piecewise syndetic subset.

In particular,
there is some sequence $n_i\to \infty, i\to\infty$ such that
\begin{equation}\label{}
  S^{n_i}z\to z,\ T^{p_1(n_i)}x\to x, \ldots,\ T^{p_d(n_i)}x\to x,\ i\to\infty.
\end{equation}
\end{thm}

In Theorem \ref{thm-main1}, if we take $(Z,S)=(X,T)$ and $z=x$, then we have the following corollary.

\begin{cor}
Let $T,S: X\rightarrow X$ be two homeomorphisms such that $(X,T)$ and $(X,S)$ are minimal. Then there is a residual subset $X_0$ of $X$ such that for any $x\in X_0$, any finite nonlinear polynomials
$p_1,\ldots, p_d$ with $p_1(0)=\cdots=p_d(0)=0$,
and any each $\ep>0$,
\begin{equation*}\label{}
 \{n\in \Z: \rho_Z(S^n x_0, x_0)<\ep \ \text{and}\ \rho_{X}(T^{p_1(n)}x_0, x_0)<\ep,\ \ldots,\ \rho_{X}(T^{p_d(n)}x_0, x_0)<\ep \}
\end{equation*}
is a piecewise syndetic subset.
\end{cor}



Using Theorem \ref{thm-main1}, we can recover some results proved in \cite{HSY22-1}. That is, we have the following immediately:

\begin{cor}[{\cite[Theorem C]{HSY22-1}}]\label{coro-ps}
Let $(X,T)$ be a minimal t.d.s. and $d\in \N$.  Let $p_{i}$ be an integral polynomial with $p_{i}(0)=0$, $1\le i\le d$. If $\deg (p_i)\ge 2, 1\le i \le d$,
then there is a dense $G_\delta$ subset $X_0$ such that for each $x\in X_0$
and each neighbourhood $U$ of $x$
$$N_{\{p_1,\ldots,p_d\}}(x,U)=\{n\in\Z: T^{p_1(n)}x\in U, \ldots, T^{p_d(n)}x\in U\}$$
is piecewise syndetic.
\end{cor}

In \cite[Theorem 4.11]{HSY22-1} we show that for a minimal t.d.s. $(X,T)$, the following statements are equivalent:
\begin{enumerate}
\item  There is a dense $G_\d$ subset $X_0$ of $X$ such that for each $x\in X_0$ and for any family $\A=\{p_1, p_2,\cdots, p_d\}$ of integral polynomials satisfying $(\spadesuit)$ (see Subsection \ref{subsection-4.2} for related definitions),  $(W_x^\A,\sigma)$ is an M-system (i.e. a transitive t.d.s. with a dense set of minimal points). 

\item There is a dense $G_\d$ subset $X_0$ of $X$ such that for each $x\in X_0$ and  for any integral polynomials $p_1,\ldots, p_d$ with $p_i(0)=0$, $1\le i\le d$, we have that for any neighbourhood $U$ of $x$
$$N_{p_1,\ldots,p_k}(x,U)=\{n\in\Z: T^{p_1(n)}x\in U, \ldots, T^{p_k(n)}x\in U\}$$
is piecewise syndetic in $\Z$.
\end{enumerate}

Thus by Corollary \ref{coro-ps}, we have
\begin{thm}[{\cite[Corollary 6.3]{HSY22-1}}]\label{}
Let $(X,T)$ be a minimal t.d.s.
There is a dense $G_\d$ subset $X_0$ of $X$ such that for each $x\in X_0$ and for any family $\A=\{p_1, p_2,\cdots, p_d\}$ of integral polynomials satisfying $(\spadesuit)$ and $\deg p_i\ge 2, 1\le i\le d$,  $(W_x^\A,\sigma)$ is an M-system. 
\end{thm}

Since for a distal system $(X,T)$ and for each point $x\in X$, $(x,x,\ldots, x)$ is a minimal pint under action $T^{a_1}\times T^{a_2}\times \cdots \times T^{a_k}$ for any $a_1,a_2,\ldots, a_k\in \Z$, by Theorem \ref{thm-main1} we have the following results for distal systems immediately.

\begin{cor}[{\cite[Theorem C]{HSY22-1}}]
Let $(X,T)$ be a minimal t.d.s. and $d\in \N$.  Let $p_{i}$ be an integral polynomial with $p_{i}(0)=0$, $1\le i\le d$. If $(X,T)$ is distal,
then there is a dense $G_\delta$ subset $X_0$ such that for each $x\in X_0$
and each neighbourhood $U$ of $x$
$$N_{\{p_1,\ldots,p_d\}}(x,U)=\{n\in\Z: T^{p_1(n)}x\in U, \ldots, T^{p_d(n)}x\in U\}$$
is piecewise syndetic.
\end{cor}




\subsection{Ideas of the proofs}\
\medskip

We use the following example to explain the basic ideas of the proof of the main results.

\begin{exam}
Let $(X,T)$ be a weakly minimal t.d.s. Then there is a residual subset $X_0$ of $X$ such that for each $x\in X_0$ and any minimal point $z\in Z$, where $(Z,S)$ is any t.d.s., there is some sequence $n_i\to\infty, i\to\infty$ such that
\begin{equation*}
  S^{n_i}z\to z,\quad T^{n_i^2}x\to x,\ i\to\infty.
\end{equation*}
\end{exam}

To show the example recall that for each $x\in X$,
$$\omega_x^\A=(\ldots, T^{p(-1)}x,  \underset{\bullet }x, T^{p(1)}x, T^{p(2)}x, \ldots)=(T^{p(n)}x)_{n\in {\Z}},$$
where $\A=\{p\}$ with $p(n)=n^2$ for any $n\in\Z$.
Since $(X,T)$ is weakly mixing, By Theorem  \cite[Theorem 1.2]{HSY-19-1} (or in Theorem \ref{thm-poly-saturate}, $X=X^*$ and $X_\infty=X_\infty^*$
is trivial), we get that there is a residual subset $X_0$ of $X$ such that for all $x\in X_0$, $\overline{\O}(\w_x^\A,\sigma)=X^\Z$.

Now let $x\in X_0$, and let $(Z,S)$ be a minimal t.d.s. and $z\in Z$. Since $\w_x^\A$ is a transitive point of $(X^\Z,\sigma)$ and $Z=\overline{\O}(z, S)$ is minimal, $\overline{\O}\Big((z,\w_x^\A), S\times \sigma\Big)$ is a joining of the t.d.s. $(Z, S)$ and $(X^\Z,\sigma)$. Since $(Z,S)$ is a minimal system, by Furstenberg disjointness theorem, 
$(X^\Z,\sigma)$ is disjoint from $(Z, S)$. It follows that
$$\overline{\O}\Big((z,\w_x^\A), S \times \sigma\Big)=Z \times X^\Z.$$
In particular, $(z,\w_x^\A)\in \overline{\O}\Big((z,\w_x^\A), S \times \sigma\Big)$.
Hence there is some sequence $n_i\to\infty, i\to\infty$ such that
\begin{equation*}
  S^{n_i}z\to z,\quad T^{n_i^2}x=T^{p(n_i)} x \to x,\ i\to\infty.
\end{equation*}

\medskip
When $(X,T)$ is not weakly mixing, its maximal $\infty$-step pro-nilfactor $X_\infty$ is not trivial. Then the proof is much more involved,
though the basic idea of the proof is similar to the weakly mixing case. Namely, in the general case, except some special situations we will show that up to an $\ep$-error, the induced dynamical system $\overline{\O}(\w_x^\A,\sigma)$ can be viewed as the product of a Bernoulli system with an $\infty$-step pronilsystem by using Theorem \ref{thm-0-dim-relative} which states that every factor map between minimal t.d.s. has an almost one to one
zero dimensional extension, the saturation theorem \cite{HSY22-1, Qiu} for polynomials and a coding process. Then we use Furstenberg disjointness theorem to complete the proof.

\subsection{Some conjectures}\
\medskip

We have the following conjectures:

\begin{conj}\label{c-1}
{\it There are minimal systems $(X,T)$ and $(X,S)$  such that $(x,x)$ is not recurrent under $T\times S$ for any $x\in X$.}
\end{conj}

We remark that without assuming minimality, there are t.d.s. $(X,T)$ and $(X,S)$  such that $(x,x)$ is not recurrent under $T\times S$ for any $x\in X$, see instance \cite[Page 40]{F}.
\begin{conj}\label{c-3}
{\it For any pair of integral polynomials $p$ and $q$ vanishing at $0$ with $\deg p\ge 2$ and $\deg q\ge 2$, there are minimal systems $(X,T)$ and $(X,S)$ such that for any $x\in X$ there is no subsequence $\{n_i\}$ of $\Z$, $n_i\ra \infty$, with $T^{p(n_i)}x\ra x$ and $S^{q(n_i)}x\ra x$.}
\end{conj}

\begin{conj}\label{c-2} {\it There are minimal systems $(X,T)$ and $(X,S)$  such that for any $x\in X$ there is no subsequence $\{n_i\}$ of $\Z$,
$n_i\ra \infty$, with $T^{n_i}x\ra x$, $T^{2n_i}x\ra x$ and $S^{n_i^2}x\ra x$.}
\end{conj}










\subsection*{Organization of the paper}

We organize the paper as follows. In Section 2, we provide necessary notions. Three main tools, i.e. the saturation theorem
of polynomials, the induced spaces and the zero dimensional almost one to one minimal extension are stated in
Sections 3-5 respectively. In Section 6, we show the mains results.
In the Appendix we give the proof of Theorem \ref{thm-0-dim-relative}.

\medskip
\noindent {\bf Acknowledgements:} We would like to thank Tomasz Downarowicz for useful discussions related to Theorem \ref{thm-0-dim-relative}. Also we thank
Jian Li and Hui Xu for useful suggestions.

\section{Preliminaries}\label{Section-pre}

In this section we give some necessary notions and some known facts used in the paper.
Note that, instead of just considering a single transformation $T$,
we will consider commuting
transformations $T_1$, $\ldots$ , $T_k$ of $X$. We only recall some basic
definitions and properties of systems for
one transformation. Extensions to the general case are
straightforward.

\subsection{Topological dynamical systems}\

\subsubsection{}

By a {\em topological dynamical system} (t.d.s.) we mean a pair $(X,T)$, where
$X$ is a compact Hausdorff space (in the paper we mainly deal with a compact metric space $X$ with a metric $\rho$) and $T:X\to X$ is a
homeomorphism. Let $(X, T)$ be a t.d.s. and $x\in X$. Then $\O(x,T)=\{T^nx: n\in \Z\}$ denotes the
{\em orbit} of $x$. A subset $A\subseteq X$ is called {\em invariant} (or {$T$-invariant}) if $TA= A$. When $Y\subseteq X$ is a closed and
invariant subset of the system $(X, T)$, we say that the system
$(Y, T|_Y)$ is a {\em subsystem} of $(X, T)$. Usually we will omit the subscript, and denote $(Y, T|_Y)$ by $(Y,T)$.
If $(X, T)$ and $(Y, S)$ are two t.d.s., their {\em product system} is the
system $(X \times Y, T\times S)$.

\medskip

When there are more than one t.d.s. involved, usually we should use different symbols to denote different transformations on different spaces, for example, $(X,T), (Y,S), (Z,H)$ etc. But when no confusing, it is convenient to use only one symbol $T$ for all transformations in all t.d.s. involved,  for example, $(X,T), (Y,T), (Z,T)$ etc.
In this paper, we {\em  use the same symbol $T$ for the transformations in all t.d.s.}

\subsubsection{}
Let $X, Y$ be compact metric spaces and $\phi: X \to Y$ be a map.  For $n \geq 2$ let
\begin{equation}\label{ite-n}
\phi^{(n)}=\underbrace{\phi\times \cdots \times \phi }_{\text{($n$ times)}}: X^n\rightarrow Y^n.
\end{equation} Thus
we write $(X^n,T^{(n)})$ for the $n$-fold product system $(X\times	\cdots \times X,T\times \cdots \times T)$.
The diagonal of $X^n$ is $$\Delta_n(X)=\{(x,\ldots,x)\in X^n: x\in X\}.$$
When $n=2$ we write	$\Delta(X)=\Delta_2(X)$.

\subsubsection{}
A t.d.s. $(X,T)$ is called {\em minimal} if $X$ contains no proper nonempty
closed invariant subsets. It is easy to verify that a t.d.s. is
minimal if and only if every orbit is dense. In a t.d.s. $(X, T)$ we say that a point $x\in X$ is {\em minimal} if $(\overline{\O(x,T)}, T)$ is minimal.
When $X$ is a metric compact space, $(X,T)$ is transitive if and only if there is some point $x\in X$ such that $\overline{\O(x,T)}=X$, which is called a {\em transitive point}.

\subsubsection{}
Let $(X,T)$ be a t.d.s. A pair $(x,y)\in X^2$ is {\it proximal} if $\inf_{n\in \Z} \rho(T^nx,T^ny)=0$; and it is {\it distal} if it is not proximal. Denote by ${\bf P}(X,T)$ the set of all proximal pairs of $(X,T)$.
A t.d.s. $(X,T)$ is
called {\it distal} if $(x,x')$ is distal whenever $x,x'\in X$ are
distinct. A point $x\in X$ is
said to be {\em distal} if whenever $y$ is in the orbit closure of
$x$ and $(x,y)$ is proximal, then $x = y$.


\subsubsection{}
For a t.d.s. $(X,T)$, $x\in X$ and $U\subseteq X$ let
$$N_T (x,U)=\{n\in \Z: T^nx\in U\}.$$
A point $x\in X$ is said to be {\em recurrent} if there is a sequence $n_i\to \infty$ such that
$T^{n_i}x\to x,$ as $ i\to\infty$.


\subsubsection{}
Generally, a {\em $G$-system} is a triple $\X=(X, G, \Pi)$ or just $(X,G)$, where $X$ is a compact Hausdorff space, $G$ is a Hausdorff topological group with the unit $e$ and $\Pi: G\times X\rightarrow X$ is a continuous map such that $\Pi(e,x)=x$ and $\Pi(s,\Pi(t,x))=\Pi(st,x)$ for all $s,t\in G$ and $x\in X$. We shall fix $G$ and suppress the action symbol. 
An analogous definition can be given if  $G$ is a semigroup. Also, the notions of transitivity, minimality and weak mixing are naturally generalized to group actions.

\subsection{Some facts about hyperspaces} \label{sub:ellis} \

\subsubsection{}
Let $X$ be a compact metric space. Let $2^X$ be the collection of nonempty closed subsets of $X$.
Let $\rho$ be the metric on $X$.
One may define a metric on $2^X$ as follows:
\begin{equation*}
\begin{split}
 \rho_H(A,B)& = \inf \{\ep>0: A\subseteq B_\ep(B), B\subseteq B_\ep(A)\}\\
 &= \max \Big\{\max_{a\in A} \rho(a,B),\max_{b\in B} \rho(b,A)\Big\},
\end{split}
\end{equation*}
where $\rho(x,A)=\inf_{y\in A} \rho(x,y)$ and $B_\ep (A)=\{x\in X: \rho(x, A)<\ep\}$.
The metric $\rho_H$ is called the {\em Hausdorff metric} of $2^X$.

Let $\{A_i\}_{i=1}^\infty$ be an arbitrary sequence of subsets of $X$. Define
$$\liminf A_i=\{x\in X: \text{for any neighbourhood $U$ of $x$, $U\cap A_i\neq \emptyset$ for all but finitely many $i$}\};$$
$$\limsup A_i=\{x\in X: \text{for any neighbourhood $U$ of $x$, $U\cap A_i\neq \emptyset$ for infinitely many $i$}\}.$$
We say that $\{A_i\}_{i=1}^\infty$ converges to $A$, denoted by $\lim_{i\to \infty} A_i=A$, if
$$\liminf A_i=\limsup A_i=A.$$
Now let $\{A_i\}_{i=1}^\infty\subseteq 2^X$ and $A\in 2^X$. Then $\lim_{i\to\infty} A_i=A$ if and only if $\{A_i\}_{i=1}^\infty $ converges to $A$ in $2^X$ with respect to the Hausdorff metric.

\subsubsection{}
Let $X,Y$ be two compact metric spaces. Let $F: Y\rightarrow 2^X$ be a map and $y\in Y$.
We say that $F$ is {\em upper semi-continuous (u.s.c.)} at $y$ if whenever $\lim y_i=y$, one has that $\limsup F(y_i)\subseteq F(y)$. If $F$ is u.s.c. at every point of $Y$, then we say that $F$ is u.s.c.
We say $F$ is {\em lower semi-continuous (l.s.c.)} at $y$ if whenever $\lim y_i=y$, one has that $\liminf F(y_i)\supset F(y)$. If $F$ is l.s.c. at every point of $Y$, then we say that $F$ is l.s.c.


It is easy to verify that $F: Y\rightarrow 2^X$ is u.s.c. at $y\in Y$ if and only if for each $\ep>0$ there exists a neighbourhood $U$ of $y$ such that $F(U)\subseteq B_\ep(F(y))$; and $F: Y\rightarrow 2^X$ is l.s.c. at $y\in Y$ if and only if for each $\ep>0$ there exists a neighbourhood $U$ of $y$ such that $F(y)\subseteq B_\ep(F(y'))$ for all $y'\in U$.

If $f: X\rightarrow Y$ is a continuous surjective map, then it is easy to verify that
$$F=f^{-1}: Y\rightarrow 2^X, y\mapsto f^{-1}(y)$$ is u.s.c. Let $(X,T)$ be t.d.s. Then the map $$F: X\rightarrow 2^X, x\mapsto \overline{\O(x,T)}$$ is l.s.c.

\medskip

We have the following well known result, for a proof see \cite[p.70-71]{Kura2} and \cite[p.394]{Kura1}, or \cite{Fort}.
\begin{thm}\label{thm-Fort}
Let $X,Y$ be compact metric spaces. If $F: Y\rightarrow 2^X$ is u.s.c. (or l.s.c.),
then the points of continuity of $F$ form a dense $G_\delta$ set in $Y$.
\end{thm}

\subsection{Factor maps}\

\subsubsection{}
A {\em factor map} $\pi: X\rightarrow Y$ between two t.d.s. $(X,T)$
and $(Y, T)$ is a continuous onto map which intertwines the
actions (i.e. $\pi\circ T= T\circ \pi$); one says that $(Y, T)$ is a {\it factor} of $(X,T)$ and
that $(X,T)$ is an {\it extension} of $(Y,T)$. The systems are said to be {\em isomorphic} if $\pi$ is bijective.

\subsubsection{}
Let $\pi: (X,T)\rightarrow (Y, T)$ be a factor map. Then
$$R_\pi=\{(x_1,x_2)\in X^2:\pi(x_1)=\pi(x_2)\}$$
is a closed invariant equivalence relation, and $Y=X/ R_\pi$.

Let $(X,T)$ and $(Y,T)$ be two t.d.s. and let $\pi: (X,T) \to (Y,T)$ be a factor map.
One says that:
\begin{itemize}
  \item $\pi$ is an {\it open} extension if it sends open sets to open sets; 
  \item $\pi$ is an {\it almost one to one} extension  if there exists a dense $G_\d$ subset $X_0\subseteq X$ such that $\pi^{-1}(\{\pi(x)\})=\{x\}$ for any $x\in X_0$;

\end{itemize}

The following is a well known fact about open mappings (see \cite[Appendix A.8]{Vr} for example).

\begin{thm}\label{thm-open}
	Let $\pi:(X,T)\rightarrow(Y,T)$ be a factor map of t.d.s. Then the map
$	\pi^{-1}:Y\rightarrow 2^X, y\mapsto \pi^{-1}(y)$
	is continuous if and only if $\pi$ is open.
\end{thm}

\medskip

Every extension of minimal flows can be lifted to an open
extension by almost one to one modifications (\cite[Theorem 3.1]{V70}, \cite[Lemma III.6]{AuG77} or \cite[Chapter VI]{Vr}). To be precise,

\begin{thm}\label{O-diagram}
For every extension $\pi:X\rightarrow Y$ of minimal t.d.s. there
exists a commutative diagram of extensions
(called the {\em O-diagram})
\begin{equation*}
\xymatrix
{
X \ar[d]_{\pi}  &  X^* \ar[l]_{\sigma}\ar[d]^{\pi*} \\
Y &  Y^*\ar[l]^{\tau}
}
\end{equation*}
with the following properties:
\begin{enumerate}
\item[(a)]
$\sigma$ and $\tau$ are almost one to one;
\item[(b)]
$\pi^*$ is an open extension;
\item[(c)]
$X^*$ is the unique minimal set in $R_{\pi
\tau}=\{(x,y)\in X\times Y^*: \pi(x)=\tau (y)\}$ and $\sigma$ and
$\pi^*$ are the restrictions to $X^*$ of the projections of
$X\times Y^*$ onto $X$ and $Y^*$ respectively.
\end{enumerate}
\end{thm}

\medskip


We sketch the construction of these factors.
Let $\pi^{-1}: Y\rightarrow 2^X, y\mapsto \pi^{-1}(y)$. Then  $\pi^{-1}$ is a u.s.c. map, and the set $Y_c$ of
continuous points of $\pi^{-1}$ is a dense $G_\d$ subset of $Y$.
Let $$\widetilde{Y}=\overline{\{\pi^{-1}(y): y\in Y\}}\ \text{and}\ Y^*=\overline{\{\pi^{-1}(y): y\in Y_c\}},$$
where the closure is taken in $2^X$. It is obvious that $Y^*\subseteq \widetilde{Y}\subseteq 2^X$. Note that
for each $A\in \widetilde{Y}$, there is some $y\in Y$ such that $A\subseteq \pi^{-1}(y)$, and hence $A\mapsto y$
define a map $\tau: \widetilde{Y}\rightarrow Y$. It is easy to verify that $\tau: (\widetilde{Y},T)\rightarrow (Y,T)$
is a factor map. One can show that if $(Y,T)$ is minimal then $(Y^*,T)$ is a minimal t.d.s. and it is the unique minimal subsystem in $(\widetilde{Y},T)$, and $\tau: Y^*\rightarrow Y$ is an almost one to one extension such that $\tau^{-1}(y)=\{\pi^{-1}(y)\}$ for all $y\in Y_c$. It can be proved that $X^*=\{(x,\tilde y)\in X \times Y^* : x \in \tilde y\}$ is a minimal subset of $X\times Y^*$. Let $\sigma: X^*\rightarrow X$ and $\pi^*: X^*\rightarrow Y^*$ be the projections. One can show that $\sigma$ is almost one to one and $\pi^*$ is open.
See \cite[Subsection 2.3]{V77} for details.


\subsection{Piecewise syndetic subsets}\
\medskip

\begin{de}
Let $(S,\cdot)$ be a semigroup. A set $A\subseteq S$ is {\em piecewise syndetic} if and only if there exists some finite subset $K\subseteq S$ such that for every finite subset $F$ of $S$, there exists $s\in S$ such that
$$F\cdot s\subseteq \bigcup_{t\in K}t^{-1} A.$$
\end{de}

Recall that a subset $S$ of $\Z$ is {\it syndetic} if it has a bounded gap,
i.e. there is $N\in \N$ such that $\{i,i+1,\cdots,i+N\} \cap S \neq
\emptyset$ for every $i \in {\Z}$. $S$ is {\it thick} if it
contains arbitrarily long runs of integers. It is easy to verify that
a subset $S$ of $\Z$ is piecewise syndetic if and only if it is an
intersection of a syndetic set with a thick set.

\medskip

Recall that a t.d.s. is a {\em M-system} if it is transitive and the set of minimal points is dense.
The following proposition characterizes piecewise syndetic recurrence (see for example \cite[Lemma 2.1]{HY05} or \cite[Theorem 3.1]{Ak97}).

\begin{prop}\label{prop-M}
Let $T_1,T_2,\ldots, T_k: X\rightarrow X$ be commuting homeomorphisms ($k\in \N$). Let $G=\langle T_1,T_2,\ldots, T_d\rangle$ be the group generated by $T_1,T_2,\ldots, T_k$. Then $(X,G)$ is an $M$-system if and only if for each (some) transitive point $x\in X$ and any neighbourhood $U$ of $x$, $$N_G(x,U)=\{(n_1,n_2,\ldots, n_k)\in \Z^k: T_1^{n_1}\cdots T_k^{n_k}x\in U\}$$ is a piecewise syndetic subset of $\Z^k$.
\end{prop}

\begin{rem}
Note that Proposition \ref{prop-M} still hold if we replace $\Z$ by $\Z_+$.
\end{rem}

\subsection{Nilmanifolds and nilsystems}\

\subsubsection{}
Let $G$ be a group. For $g, h\in G$ and $A,B \subseteq G$, we write $[g, h] =
ghg^{-1}h^{-1}$ for the commutator of $g$ and $h$ and
$[A,B]$ for the subgroup spanned by $\{[a, b] : a \in A, b\in B\}$.
The commutator subgroups $G_j$, $j\ge 1$, are defined inductively by
setting $G_1 = G$ and $G_{j+1} = [G_j ,G]$. Let $d \ge 1$ be an
integer. We say that $G$ is {\em $d$-step nilpotent} if $G_{d+1}$ is
the trivial subgroup.

\subsubsection{}
Let $d\in\N$, $G$ be a $d$-step nilpotent Lie group and $\Gamma$ be a discrete
cocompact subgroup of $G$. The compact manifold $X = G/\Gamma$ is
called a {\em $d$-step nilmanifold}. The group $G$ acts on $X$ by
left translations and we write this action as $(g, x)\mapsto gx$.
The Haar measure $\mu$ of $X$ is the unique Borel probability measure on
$X$ invariant under this action. Fix $t\in G$ and let $T$ be the
transformation $x\mapsto t x$ of $X$, i.e. $t(g\Gamma)=(tg)\Gamma$ for each $g\in G$. Then $(X, \mu, T)$ is
called a {\em $d$-step nilsystem}. In the topological setting we omit the measure
and just say that $(X,T)$ is a $d$-step nilsystem. For more details on nilsystems, refer to \cite[Chapter 11]{HK18}.






\subsubsection{}
We will need to use inverse limits of nilsystems, so we recall the
definition of a sequential inverse limit of systems. If
$(X_i,T_i)_{i\in \N}$ are systems 
and $\pi_i: X_{i+1}\rightarrow X_i$ is a factor map for all $i\in\N$, the {\em inverse
limit} of these systems is defined to be the compact subset of
$\prod_{i\in \N}X_i$ given by $\{ (x_i)_{i\in \N }: \pi_i(x_{i+1}) =
x_i, i\in\N\}$, and we denote it by
$\lim\limits_{\longleftarrow}(X_i,T_i)_{i\in\N}$.
It is a compact metric space.
We note that the maps $\{T_i\}_{i\in \N}$ induce naturally a transformation $T$ on the inverse
limit such that $T(x_1,x_2,\ldots)=(T_1(x_1),T_2(x_2),\ldots)$.




\begin{de} [{\cite[Definition 1.2]{HKM}}] 
For $d\in \N$, a minimal t.d.s. $(X,T)$ is called  a {\em $d$-step pro-nilsystem} or {\it system of order $d$} if $X$ is an inverse limit of $d$-step minimal nilsystems.
\end{de}

\subsection{Regionally proximal relation
of order $d$}\
\medskip

\begin{de}[{\cite[Definition 3.2]{HKM}}]
Let $(X, T)$ be a t.d.s. and let $d\in \N$. The points $x, y \in X$ are
said to be {\em regionally proximal of order $d$} if for any $\d  >
0$, there exist $x', y'\in X$ and a vector ${\bf n} = (n_1,\ldots ,
n_d)\in\Z^d$ such that $\rho (x, x') < \d, \rho (y, y') <\d$, and $$
\rho (T^{{\bf n}\cdot \ep}x', T^{{\bf n}\cdot \ep}y') < \d\
\text{for every  $\ep=(\ep_1,\ldots,\ep_d)\in \{0,1\}^d\setminus\{(0,0,\ldots , 0)\}$},$$
where ${\bf n}\cdot \ep=\sum_{i=1}^d n_i\ep_i$.
The set of regionally proximal pairs of
order $d$ is denoted by $\RP^{[d]}$ (or by $\RP^{[d]}(X,T)$ in case of
ambiguity), and is called {\em the regionally proximal relation of
order $d$}.
\end{de}


The following theorems were proved in \cite{HKM} (for minimal distal systems) and
in \cite{SY} (for general minimal systems).

\begin{thm}\label{thm-1}
Let $(X, T)$ be a minimal t.d.s. and let $d\in \N$. Then
\begin{enumerate}

\item $\RP^{[d]}$ is an equivalence relation;

\item $(X,T)$ is a $d$-step pro-nilsystem if and only if $\RP^{[d]}=\Delta_X$.
\end{enumerate}
\end{thm}

The regionally proximal relation of order $d$ allows us to construct the maximal $d$-step
pro-nilfactor of a system.

\begin{thm}\label{thm0}
Let $\pi: (X,T)\rightarrow (Y,T)$ be a factor map between minimal t.d.s.
and let $d\in \N$. Then
\begin{enumerate}
  \item $\pi\times \pi (\RP^{[d]}(X,T))=\RP^{[d]}(Y,T)$;
  \item $(Y,T)$ is a $d$-step pro-nilsystem if and only if $\RP^{[d]}(X,T)\subseteq R_\pi$.
\end{enumerate}
In particular, $X_d\triangleq X/\RP^{[d]}(X,T)$, the quotient of $(X,T)$ under $\RP^{[d]}(X,T)$, is the
maximal $d$-step pro-nilfactor of $X$. 
\end{thm}

\subsection{$\infty$-step pro-nilsystems}\
\medskip

By Theorem \ref{thm-1} for any minimal t.d.s. $(X,T)$,
$$\RP^{[\infty]}=\bigcap\limits_{d\ge 1} \RP^{[d]}$$
is a closed invariant equivalence relation (we write $\RP^{[\infty]}(X,T)$ in case of ambiguity). Now we formulate the
definition of $\infty$-step pro-nilsystems. 

\begin{de}[{\cite[Definition 3.4]{D-Y}}]
A minimal t.d.s. $(X, T)$ is an {\em $\infty$-step
pro-nilsystem} or {\em a system of order $\infty$}, if the equivalence
relation $\RP^{[\infty]}$ is trivial, i.e., coincides with the
diagonal.
\end{de}

\begin{rem}
\begin{enumerate}
  \item Similar to Theorem \ref{thm0}, one can show that the quotient of a
minimal system $(X,T)$ under $\RP^{[\infty]}$ is the maximal
$\infty$-step pro-nilfactor of $(X,T)$. We denote the maximal
$\infty$-step pro-nilfactor of $(X,T)$ by $X_\infty$.

  \item A minimal system is an $\infty$-step pro-nilsystem if and only if it is
an inverse limit of minimal nilsystems \cite[Theorem 3.6]{D-Y}.

\end{enumerate}
\end{rem}

\subsection{Symbolic dynamics}

Let $\Sigma$ be a finite alphabet with $m$ symbols, $m \ge 2$. We usually
suppose that $\Sigma=\{0,1,\cdots,m-1\}$. Let $\Omega=\Sigma^{\Z}$ be the
set of all sequences $x=\ldots x_{-1}x_0x_1 \ldots=(x_i)_{i\in \Z}$, $x_i \in
\Sigma$, $i \in \Z$, with the product topology. A metric compatible is
given by $d(x,y)=\frac{1}{1+k}$, where $k=\min \{|n|:x_n \not= y_n
\}$, $x,y \in \Omega$. The shift map $\sigma: \Omega \longrightarrow \Omega$
is defined by $(\sigma x)_n = x_{n+1}$ for all $n \in \Z$. The
pair $(\Omega,\sigma)$ is called a {\em shift dynamical system}. Any subsystem of $(\Omega, \sigma)$
is called a {\em subshift system}.
Similarly we can replace $\Z$ by $\Z_+=\{0,1,2,\ldots \}$, and $\sigma$ will be not a
homeomorphism but a surjective map.

Each element of $\Sigma^{\ast}= \bigcup_{k \ge 1} \Sigma^k$ is called {\em a word} or {\em a block} (over $\Sigma$).
We use $|A|=n$ to denote the length of $A$ if $A=a_1\ldots a_n$.
If $\omega=(\cdots \omega_{-1} \omega_0 \omega_1 \cdots) \in \Omega$ and $a \le b \in \Z$, then
$\omega[a,b]=:\omega_{a} \omega_{a+1} \cdots \omega_{b}$ is a $(b-a+1)$-word occurring in
$\omega$ starting at place $a$ and ending at place $b$.
Similarly we define $A[a,b]$ when $A$ is a word. A word $A$ {\em appears} (or {\em occurs}) in the word $B$ if there are some $a\le b$ such that
$B[a,b]=A$.

For $n\in\N$ and words $A_1,\ldots, A_n$, we denote by $A_1\ldots A_n$ the concatenation of $A_1,\ldots, A_n$.
When $A_1=\cdots=A_n=A$ denote $A_1\ldots A_n$ by $A^n$. 
If $(X,\sigma)$ is a subshift system, let $[i]=[i]_X=\{x\in X:x(0)=i\}$ for $i\in \Sigma$, and
$[A]=[A]_X=\{x\in X:x_0 x_1\cdots x_{(|A|-1)}=A\}$ for any word $A$.

\section{Disjointness and Saturation theorems for polynomials}\label{section-saturate}

In this section, we introduce two important ingredients in the proof of the main results.

\subsection{Disjointness}\
\medskip

The notion of {\em disjointness} of two dynamical systems was introduced by
Furstenberg in his seminal paper \cite{F67} both in ergodic theory and topological dynamics.
This notion plays an important role in ergodic system, see \cite{G} for instance.

Let $(X,T)$ and $(Y,S)$ be
two t.d.s. We say that $J\subset X\times Y$ is a {\em joining} of $X$
and $Y$ if $J$ is a nonempty closed invariant set, and it projects
onto $X$ and $Y$ respectively. If each joining is equal to $X\times
Y$ then we say that $(X,T)$ and $(Y,S)$ are {\em disjoint}, denoted
by $(X,T)\perp (Y,S)$ or $X\perp Y$. Note that if $(X,T)\perp (Y,S)$
then one of them is minimal \cite{F67}, and if in addition $(Y,S)$ is minimal then
the set of recurrent points of $(X,T)$ is dense in $X$ \cite{HY05}.

In \cite{F67}, Furstenberg showed that each totally transitive system
with dense set of periodic points is disjoint from all minimal
systems; and each weakly mixing system is disjoint from all minimal
distal systems. He left the following question
\cite[Problem G]{F67}: Describe the system who is disjoint
from all minimal systems.
In \cite{HSY20}, the authors provide a resolution to Furstenberg's problem concerning transitive systems.

\subsection{Saturation theorems}\
\medskip

Let $X, Y$ be sets, and let $\phi : X\rightarrow Y$ be a map. A subset $L$ of $X$ is called
{\em $\phi$-saturated} if $$\{x\in L: \phi^{-1}(\phi(x))\subseteq L\}=L,$$ i.e., $L=\phi^{-1}(\phi(L))$.

Let $(X,T)$ be a t.d.s. and $d\in \N$. Let $\A=\{p_1,\ldots, p_d\}$ be a family of integral polynomials and define
$$\O_\A((x_1,\ldots,x_d))=\O_{\{p_1,\ldots,p_d\}}((x_1,\ldots,x_d))=\{(T^{p_1(n)}x_1, \ldots, T^{p_d(n)}x_d): n\in \Z \},$$
and
$$\overline{\O}_\A((x_1,\ldots,x_d))=\overline{\{(T^{p_1(n)}x_1, \ldots, T^{p_d(n)}x_d): n\in \Z \}}.$$

\medskip

By the proof of \cite[Theorem 3.6]{HSY22-1}, we have

\begin{thm}\label{thm-poly-sat1}
Let $(X,T)$ be a minimal t.d.s., and $\pi:X\rightarrow X_\infty$ be the factor map from $X$ to its maximal $\infty$-step pro-nilfactor $X_\infty$.
Let $\pi:(X,T)\rightarrow (Y,T)$ be an extension of minimal t.d.s. such that $\pi$ is open and $X_{\infty}$ is a factor of $Y$.
$$
\xymatrix{
                &         X \ar[d]^{\pi} \ar[dl]_{\pi_\infty}    \\
  X_{\infty} & Y   \ar[l]_{\phi}          }
$$
Then there is a $T$-invariant residual subset $X_0$ of $X$ having the following property:  for all $x\in X_0$, all sets of essentially distinct
non-constant integral polynomials $\A=\{p_1,\ldots, p_d\}$ and $d\in \N$, $\overline{\O}_\A(x^{\otimes d})$ is ${\pi}^{(d)}$-saturated, that is,
$$\overline{\O}_\A(x^{\otimes d})=({\pi}^{(d)})^{-1}\Big({\pi}^{(d)}(\overline{\O}_\A(x^{\otimes d}))\Big).$$
\end{thm}

Using this result and O-diagram, one has the following Theorem.

\begin{thm}[{\cite[Theorem 3.6]{HSY22-1}}]
Let $(X,T)$ be a minimal t.d.s., and $\pi:X\rightarrow X_\infty$ be the factor map from $X$ to its maximal $\infty$-step pro-nilfactor $X_\infty$.
Then there are minimal t.d.s. $X^*$ and $X_\infty^*$ which are almost one to one
extensions of $X$ and $X_\infty$ respectively, an open factor map $\pi^*$ and a commuting diagram below
\[
\begin{CD}
X @<{\varsigma^*}<< X^*\\
@VV{\pi}V      @VV{\pi^*}V\\
X_\infty @<{\varsigma}<< X_\infty^*
\end{CD}
\]
such that there is a $T$-invariant residual subset $X^*_0$ of $X^*$ having the following property:  for all $x\in X^*_0$, all sets of essentially distinct
non-constant integral polynomials $\A=\{p_1,\ldots, p_d\}$ and $d\in \N$, $\overline{\O}_\A(x^{\otimes d})$ is ${\pi^*}^{(d)}$-saturated.

If in addition $\pi$ is open, then $X^*=X$, $X_\infty^*=X_\infty$ and $\pi^*=\pi$.
\end{thm}


Let $(X,T)$ be a t.d.s. and $\A=\{p_1,\ldots, p_d\}$ be a family of integral polynomials. Put
\begin{equation}\label{nn}
  N_\A(X)=\overline{\bigcup_{x\in X}\O_{\A}(x^{\otimes d})}=\overline{\{(T^{p_1(n)}x, \ldots, T^{p_d(n)}x): x\in X, n\in \Z \}}.
\end{equation}
The special case of the following theorem when $\A=\{n,2n, \ldots, dn\}$ has been considered in \cite[Lemma 3.3]{GHSWY}.

By the proof of \cite[Theorem 3.10]{HSY22-1}, we have

\begin{thm}\label{thm-poly-sat2}
Let $(X,T)$ be a minimal t.d.s., and $\pi:X\rightarrow X_\infty$ be the factor map from $X$ to its maximal $\infty$-step pro-nilfactor $X_\infty$.
Let $\pi:(X,T)\rightarrow (Y,T)$ be an extension of minimal t.d.s. such that $\pi$ is open and $X_{\infty}$ is a factor of $Y$.
$$
\xymatrix{
                &         X \ar[d]^{\pi} \ar[dl]_{\pi_\infty}    \\
  X_{\infty} & Y   \ar[l]_{\phi}          }
$$
Then
for all sets of essentially distinct non-constant integral polynomials $\A=\{p_1,\ldots, p_d\}$, $N_\A(X)$ is ${\pi}^{(d)}$-saturated, that is,
$$N_\A(X)=({\pi}^{(d)})^{-1}\Big({\pi}^{(d)}(N_\A(X))\Big)=({\pi}^{(d)})^{-1}
\Big(N_\A(X_\infty)\Big).$$
\end{thm}

\section{The induced systems $N_\infty(X,\A)$ and $M_\infty(X,\A)$ for polynomials}\label{Section-Def-M}

In this section, for a given finite set $\A$ of integral polynomials and a t.d.s. $(X,T)$, we define $\Z^2$- t.d.s.
$N_\infty(X,\A)$ and $M_\infty(X,\A)$, which were introduced in \cite{HSY22-1} and \cite{HSY-new}.

\subsection{The definition of $N_\infty(X,\A)$ and  $M_\infty(X,\A)$}\

\subsubsection{Notation}

Let $d\in\N$ and $\A=\{p_1, p_2,\ldots, p_d\}$ be a family of integral polynomials with $p_i(0)=0$, $1\le i\le d$.
A point of $(X^d)^{\Z}$ is denoted by
$${\bf x}=({\bf x}_n)_{n\in {\Z}}=\Big((x^{(1)}_n, x^{(2)}_n,\ldots, x^{(d)}_n) \Big)_{n\in \Z}.$$

Let $\vec{p}=(p_1,p_2,\cdots, p_d)$ and let $T^{\vec{p}(n)}: X^d\rightarrow X^d$ be defined by
\begin{equation}\label{}
  T^{\vec{p}(n)}(x_1,x_2,\ldots, x_d)=(T^{p_1(n)}x_1, T^{p_2(n)}x_2,\ldots, T^{p_d(n)}x_d).
\end{equation}

Recall that $T^{(d)}=T\times T\times \cdots\times T$ ($d$-times).
Define $T^\infty: (X^d)^{{\Z}}\rightarrow (X^d)^{{\Z}}$ by
$$T^\infty({\bf x}_n)_{n\in {\Z}}=(T^{(d)}{\bf x}_n)_{n\in {\Z}}.$$
Let $\sigma: (X^d)^{{\Z}}\rightarrow (X^d)^{{\Z}}$
be the shift map, i.e.,  for all $({\bf x}_n)_{n\in {\Z}}\in (X^d)^{{\Z}}$
$$(\sigma {\bf x})_n={\bf x}_{n+1}, \ \forall n\in {\Z}. $$

\subsubsection{Definition of $N_\infty(X,\A)$}
Let $x^{\otimes d}=(x,x,\ldots, x)\in X^d$ and
$$\D_{\infty}(X)=\{x^{(\infty)}\triangleq (\ldots, x^{\otimes d}, x^{\otimes d},\ldots ) \in (X^d)^{{\Z}}: x\in X\}.$$
For each $x\in X$, put
\begin{equation}\label{}
  \w_x^\A\triangleq(T^{\vec{p}(n)}x^{\otimes d})_{n\in \Z}=\big ((T^{p_1(n)}x, T^{p_2(n)}x,\ldots, T^{p_d(n)}x) \big)_{n\in \Z}
  \in (X^d)^{{\Z}},
\end{equation}
and set
\begin{equation}\label{}
  N_\infty(X,\A)=\overline{\bigcup\{\O(\w_x^\A,\sigma): x\in X\}}\subseteq (X^d)^{{\Z}}.
\end{equation}

\begin{rem} We have the following
\begin{enumerate}

\item It is clear that $N_\infty(X,\A )$ is invariant under the action of $T^\infty$ and $\sigma$, and $T^\infty\c \sigma=\sigma\c T^\infty$. Thus
$(N_\infty(X,\A), \langle T^\infty, \sigma\rangle)$ is a $\Z^2$-t.d.s.

\item If $(X,T)$ is transitive, then for each transitive point $x$ of $(X,T)$, $N_\infty(X,\A)=\overline{\O(\w_x^\A, \langle T^\infty, \sigma\rangle )}$.

\item Sometimes we identify  points in $(X^{d_1+d_2})^{\Z}$ as $(X^{d_1})^{\Z}\times (X^{d_2})^{\Z}$ as follows: 

{\small $$\Big((x^{(1)}_n, \ldots, x^{(d_1+d_2)}_n) \Big)_{n\in \Z}=\Big(\big((x^{(1)}_n, \ldots, x^{(d_1)}_n) \big)_{n\in \Z}, \big((x^{(d_1+1)}_n, \ldots, x^{(d_1+d_2)}_n) \big)_{n\in \Z}\Big).$$}


\item Similarly, we may define systems in $(X^{d})^{\Z_+}$ associated with the polynomials $\A$. In such a case,
$\sigma$ is a continuous surjective map but not a homeomorphism.
\end{enumerate}
\end{rem}

\subsubsection{Definition of $M_\infty(X,\A)$}
Assume that $\A=\{p_1, \ldots, p_{s}, p_{s+1}, \ldots, p_{d}\}$, where $p_i(n)=a_in, 1\le i\le s$ with
$s\ge 0$, $a_1,a_2,\ldots, a_s\in \Z\setminus\{0\}$ are distinct, and $\deg p_{i}\ge 2, s+1\le i\le d$.
Note that $s=0$  means that $\A=\{p_1, \ldots, p_{d}\}$ with $\deg p_i\ge 2$, $1\le i\le d$.

We have that
$$\omega_x^\A =\Big((T^{a_1n}x, \ldots, T^{a_{s}n}x, T^{p_{s+1}(n)}x,\ldots, T^{p_d(n)}x)\Big)_{n\in {\Z}}\in (X^d)^{\Z}.$$

Define $\widetilde{\sigma}: X^s\times (X^{d-s})^{\Z} \rightarrow X^s\times (X^{d-s})^{\Z}$ as follow:
for $((x_1,\ldots, x_s), {\bf x})\in X^s\times (X^{d-s})^{\Z}$, let
$$\widetilde{\sigma}\Big(\big((x_1,x_2, \ldots, x_s), {\bf x}\big)\Big)=\Big(\big((T^{a_1}x_1, T^{a_2}x_2, \ldots, T^{a_s}x_s), \sigma'{\bf x}\big)\Big),$$
where $\sigma'$ is the shift on $(X^{d-s})^{\Z}$. Recall that $\tau_{\vec{a}}=T^{a_1}\times \cdots \times T^{a_s}$. So, $\widetilde{\sigma}=\tau_{\vec{a}}\times \sigma'$.
Let
\begin{equation}\label{xi}
  \xi_x^\A=\Big((x, \ldots , x), (T^{p_{s+1}(j)}x,\ldots, T^{p_d(j)}x)_{j\in {\Z}}\Big)\in X^s\times (X^{d-s})^{\Z}.
\end{equation}
Then for $n\in {\Z}$,
\begin{equation}\label{w3}
  \widetilde {\sigma}^n\xi_x^\A =\Big((T^{a_1n }x, \ldots , T^{a_sn}x), (T^{p_{s+1}(n+j)}x,\ldots, T^{p_d(n+j)}x)_{j\in {\Z}}\Big)\in X^s\times (X^{d-s})^{\Z}.
\end{equation}

We set
\begin{equation}\label{}
  M_\infty(X,\A)=\overline{\bigcup\{\O(\xi_x^\A,\widetilde{\sigma}): x\in X\}}\subseteq X^s\times (X^{d-s})^{\Z}.
\end{equation}

It is clear that $M_\infty(X,\A )$ is invariant under the action of $T^\infty$ and $\widetilde{\sigma}$, and $(M_\infty(X,\A), \langle T^\infty, \widetilde{\sigma}\rangle)$ is a t.d.s.

\subsubsection{}

We define
$$\phi:(M_\infty(X,\A), \widetilde{\sigma})\rightarrow (N_\infty(X,\A),\sigma), \ ({\bf y}, {\bf x})\mapsto\Big( (\ldots,  \tau^{-1}_{\vec{a}} {\bf y}, \underset{\bullet } {\bf y}, \tau_{\vec{a}} {\bf y}, \tau_{\vec{a}}^2 {\bf y},\ldots ), {\bf x}\Big),$$ and it is an isomorphism. It is easy to show that

\begin{lem}[{\cite[Lemma 4.2.]{HSY22-1}\label{M-N-equ}}]
For any t.d.s. $(X,T)$ we have
$$(M_\infty(X,\A),\langle T^\infty, \widetilde{\sigma}\rangle)\cong (N_\infty(X,\A), \langle T^\infty, \sigma\rangle),$$ and for each $x\in X$
$$\big(\overline{\O}(\w_x^\A,\sigma),\sigma\big)\cong \big(\overline{\O}(\xi_x^\A,\widetilde{\sigma}), \widetilde{\sigma}\big).$$
\end{lem}

In the sequel when there is no room for confusion, we will use $\sigma$ instead of $\widetilde{\sigma}$
when studying $(M_\infty(X,\A),\langle T^\infty, \widetilde{\sigma}\rangle)$.

\subsection{The condition $(\spadesuit)$}\label{subsection-4.2}\
\medskip

For a given integral polynomial $p$ with $p(0)=0$, for each $j\in\Z$ let $p^{[j]}$ be defined by
$$p^{[j]}(n)=p(n+j)-p(j), \ \forall n\in\Z.$$

\begin{de}
Let $\A=\{p_1, p_2, \ldots, p_{d}\}$ be a family of integral polynomials. We say $\A$ satisfies {\em condition $(\spadesuit)$} if $p_1(0)=\cdots =p_d(0)=0$ and
\begin{enumerate}
  \item $p_i(n)=a_i n, 1\le i\le s$, where $s\ge 0$, and $a_1,a_2,\ldots, a_s$ are distinct non-zero integers;
  \item $\deg p_{j}\ge 2, s+1\le j\le d$;
  \item for each $i\neq j\in \{s+1,s+2,\ldots, d\}$, $p_j^{[k]}\neq p_i^{[t]}$ for any $k,t\in \Z$.
\end{enumerate}
\end{de}

\begin{lem}\label{ww=dem}
Let $k\in\N$ and $Q=\{q_1,q_2,\ldots,q_k\}$ be a family if non-constant integral polynomials with $q_i(0)=0, 1\le i\le k$. Then there is $1\le d\le k$
and $\A=\{p_1,\ldots,p_d\}\subset \{q_1,\ldots,q_k\}$ such that $\A$ satisfies condition $(\spadesuit)$, and $\{q_1, \ldots, q_k\}\subseteq \{p_i^{[m]}: 1\le i\le d, m\in \Z \}$.

\end{lem}

\begin{thm}[{\cite[Theorem 4.8.]{HSY22-1}}]
Let $(X,T)$ be a t.d.s. and $\A'$ be a family of finitely many non-constant integral polynomials with $q(0)=0, \forall q\in \A'$. Then there is a family of integral polynomials $\A$ satisfying the condition $(\spadesuit)$ such that
$$(N_\infty(X,\A),\langle T^\infty, {\sigma}\rangle)\cong (N_\infty(X,\A'), \langle T^\infty, \sigma\rangle),$$
and
$$(M_\infty(X,\A),\langle T^\infty, {\widetilde{\sigma}}\rangle)\cong (M_\infty(X,\A'), \langle T^\infty, \widetilde{\sigma}\rangle).$$
\end{thm}

\subsubsection{}
Let $(X,T)$ be a t.d.s. Assume that $\A=\{p_1, \ldots, p_{s}, p_{s+1}, \ldots, p_{d}\}$, where $p_i(n)=a_in, 1\le i\le s$ with
$s\ge 0$, $a_1,a_2,\ldots, a_s\in \Z\setminus\{0\}$ are distinct, and $\deg p_{i}\ge 2, s+1\le i\le d$. Let $\xi_x^\A$ be defined in \eqref{xi}. Let
\begin{equation}\label{}
  W_{x}^\A=\overline{\O}(\xi_{x}^{\A}, \sigma).
\end{equation}
Then $(W_{x}^\A,\sigma)$ is transitive and $\xi_{x}^\A$ is a transitive point.
Similar to the saturation theorems proved in the previous section, we have


By the proof of \cite[Theorem 4.9]{HSY22-1}, we have:

\begin{thm}\label{thm-poly-saturate}
Let $(X,T)$ be a minimal t.d.s., and $\pi:X\rightarrow X_\infty$ be the factor map from $X$ to its maximal $\infty$-step pro-nilfactor $X_\infty$.
Let $\pi:(X,T)\rightarrow (Y,T)$ be an extension of minimal t.d.s. such that $\pi$ is open and $X_{\infty}$ is a factor of $Y$.
$$
\xymatrix{
                &         X \ar[d]^{\pi} \ar[dl]_{\pi_\infty}    \\
  X_{\infty} & Y   \ar[l]_{\phi}          }
$$
Then
we have the following properties:
\begin{enumerate}
  \item There is a residual subset $X_0$ of $X$  such that for all $x\in X_0$ and all family $\A=\{p_1,\ldots, p_d\}$ satisfying $(\spadesuit)$, $W_{x}^\A$ is ${\pi}^{\infty}$-saturated (${\pi}^{\infty}$ is defined similar to $T^\infty$), that is,
$$W_{x}^\A=({\pi}^{\infty})^{-1}\Big({\pi}^{\infty}(W_{x}^\A)\Big).$$


  \item For each family $\A=\{p_1,\ldots, p_d\}$ satisfying $(\spadesuit)$, $M_\infty(X,\A)$ is ${\pi}^\infty$-saturated, that is,
  $$M_\infty(X,\A)=({\pi}^\infty)^{-1}\Big(M_\infty(X_\infty,\A)\Big ).$$
\end{enumerate}

\end{thm}

\subsection{Nilsystems}\
\medskip

For nilsystems, we have the following result.

\begin{thm}[{\cite[Theorem 5.5]{HSY22-1}}]\label{thm-nil-ergodic}
Let $(X, T)$ be a minimal pro-nilsystem. 
Let $\A= \{p_1,  p_2, \ldots,  p_d \}$ be a family of non-constant essentially distinct integral polynomials with $p_1(0)=\cdots =p_d(0)=0$. Then we have
\begin{enumerate}
  \item The system $(N_\infty(X,\A), \langle T^{\infty},\sigma \rangle)$ is a minimal pro-nilsystem.

  \item For each $x\in X$, the system $(\overline{\O}(\omega_x^\A, \sigma),\sigma)$ is a
minimal pro-nilsystem.
\end{enumerate}
\end{thm}

Note that Theorem \ref{thm-nil-ergodic} remains true for $\infty$-step pro-nilsystems.

\section{Zero-dimensional extensions}

In this section, we show that each factor map has a zero dimensional extension, which will be used in the proof of main results. To be precise, we will show
\begin{thm}\label{thm-0-dim-relative}
Let $\pi: (X,T)\rightarrow (Y,S)$ be an extension of t.d.s. $(X,T)$ and $(Y,S)$. Then
there exists a commutative diagram of extensions
\begin{equation*}
\xymatrix
{
(X,T) \ar[d]_{\pi} & (X',T') \ar[l]_{\tau}\ar[d]^{\pi'} \\
(Y,T) &  (Y',S') \ar[l]^{\theta}
}
\end{equation*}
with the following properties:
\begin{enumerate}
\item
$\theta$ and $\tau$ are almost one to one;
\item
$(X',T')$ and $(Y',S')$ are zero dimensional t.d.s.
\end{enumerate}

If in addition $(X,T)$ is minimal, then $(X',T')$ and $(Y',S')$ are minimal.
\end{thm}

Experts in the field may know this result, yet it remains undocumented in the literature. To ensure completeness, we provide a proof in the appendix.

\section{Proof of Theorem \ref{thm-main1}}

In this section we give the proof of Theorem \ref{thm-main1}.

\subsection{A special case}\
\medskip

First we deal with a special case.




\begin{lem}\label{lem7}
Let $(X,T)$ be a minimal t.d.s. and $\A=\{p_1,\ldots, p_d\}$ be a family of integer polynomials with $p_1(0)=\cdots =p_d(0)=0$. If $(X,T)$ is an almost one to one extension of its maximal $\infty$-step pro-nilfactor $X_\infty$, then there is a $T$-invariant residual subset $X_0$ of $X$ such that for each $x\in X_0$, $\w_x^\A$ is a distal point of $\big((X^d)^\Z,\sigma\big)$.

In particular, for each $x_0\in X_0$, any minimal point $z\in Z$, where $(Z,S)$ is any t.d.s.,
and any each $\ep>0$,
\begin{equation*}\label{}
 N_{z,x_0,\ep}=\{n\in \Z: \rho_Z(S^n z,z)<\ep \ \text{and}\ \rho_{X}(T^{p_1(n)}x_0, x_0)<\ep,\ \ldots,\ \rho_{X}(T^{p_d(n)}x_0, x_0)<\ep \}
\end{equation*}
is a syndetic subset.
\end{lem}

\begin{proof}
Let $\pi:X\rightarrow  X_\infty$ be the almost one to one extension. Let $X_0$ be the set of points of $X$ satisfying
\begin{enumerate}
\item for each $x\in X_0$, $\pi^{-1}\pi(x)=\{x\}$, and $X_0$ is $T$-invariant.

\item for each $x\in X_0$, $(x,\ldots, x)$ ($d$-times) is a recurrent point with respect to $\A$. 
\end{enumerate}
It is clear that $X_0$ is $T$-invariant and residual, since both points satisfying (1) or (2) form a $T$-invariant dense $G_\delta$ sets (this follows by the definition
and \cite{BL96}).

For $x\in X_0$, let $y=\pi(x)$. Then $(\pi^\infty)^{-1}(\w_y^\A)=\{\w_x^\A\}$.
Thus $\w_x^\A$ is a distal point for $\sigma$ since $\w_y^\A$ is a distal point for $\sigma$ (by Theorem \ref{thm-nil-ergodic}) and $(\pi^\infty)^{-1}(\w_y^\A)=\{\w_x^\A\}$. Hence $\left(\overline{\O}(\w_x^\A,\sigma),\sigma\right)$ is minimal and $$\pi^\infty:\overline{\O}(\w_x^\A,\sigma)\rightarrow \overline{\O}(\w_y^\A,\sigma)$$ is an almost one to one extension of two minimal t.d.s.

Since $\w_x^\A$ is a distal point for $\sigma$, $(\w_x^\A, z)$ is a minimal point for each minimal point $z$ in a minimal system $(Z,S)$ under the
action of $\sigma\times S$ by \cite[Theorem 9.11]{F}. The proof is complete.
\end{proof}

\subsection{Proof of Theorem \ref{thm-main1} when the polynomials satisfying condition $(\spadesuit)$} \label{subsec7.3}
\
\medskip

Before the proof of Theorem \ref{thm-main1}, we describe the basic ideas. Let $(X,T)$ be a minimal t.d.s., $\pi:X\ra X_\infty$ is not almost one to one, and $\A$ be a finite family of nonlinear polynomials satisfying $(\spadesuit)$.

First we use Theorem \ref{thm-0-dim-relative} to get a zero dimensional extension $\pi':X'\ra Y'=X_\infty'$ of $\pi:X\ra X_\infty$, and then use Theorem \ref{thm-poly-saturate} to get an open extension $\pi^*:X^*\ra Y^*$ of $\pi'$. Since $X'$ and $Y'$ are zero dimensional, this allows us to construct a clopen partition of $X'$ and a clopen partition of $X^*$ for a given $\ep>0$ (note that $X^*$ is not necessarily zero dimensional). Then we show that up to an $\ep$-error, the induced dynamical system can be viewed as the product a Bernoulli system with an almost $\infty$-step pronilsystem.
Finally we use Furstenberg disjointness theorem to conclude the proof.

\begin{proof}
By Lemma \ref{lem7}, it remains to show the case when $(X,T)$ is not an almost one to one extension of $X_\infty$.
Without loss of generality we may assume that the maximal equicontinuous factor of $X$ is infinite. In fact, if the maximal equicontinuous factor of $X$ is finite, then
we may replace $X$ by the product system $(X\times {\mathbb S}^1, T\times S)$, where $({\mathbb S}^1,S)$ is an irrational rotation on the unit circle. Since $(X,T)$ and  $({\mathbb S}^1,S)$ has no nontrivial common factors, $(X\times {\mathbb S}^1, T\times S)$ is minimal \cite[Corollary 22, Chapter 11]{Au88}, and its maximal equicontinuous factor is infinite.

\medskip

We divide the proof into the following steps.


\medskip

\noindent {\bf Step 1.} {\em Notations and preparation.}

\medskip

Let $(X,T)$ be a minimal t.d.s., and $\pi:X\rightarrow X_\infty$ be the factor map from $X$ to its maximal $\infty$-step pro-nilfactor $X_\infty$. By Theorem \ref{thm-0-dim-relative}, there are almost one to one extensions $\tau: X'\rightarrow  X$ and $\theta: Y'\rightarrow X_\infty$
such that $(X', T)$ and $(Y',T)$ are zero dimensional minimal t.d.s.
\begin{equation*}
\xymatrix
{
X \ar[d]_{\pi} & X' \ar[l]_{\tau}\ar[d]^{\pi'} \\
X_\infty &  Y' \ar[l]^{\theta}
}
\end{equation*}
By Theorem \ref{thm-poly-saturate} there are minimal t.d.s. $X^*$ and $Y^*$ which are almost one to one
extensions of $X'$ and $Y'$ respectively, an open factor map $\pi^*$ and a commuting diagram below.
\begin{equation*}
	\xymatrix
	{X' \ar[d]_{\pi'}
		 &
		X^* \ar[l]_{\varsigma^*}
		\ar[d]^{\pi^*}
           \\
		Y'          &
		Y^* \ar[l]^{\varsigma}
			}
	\end{equation*}
By Theorem \ref{thm-poly-saturate} we have the following properties:
\begin{enumerate}
  \item There is a residual invariant subset $X_0^*$ of $X^*$  such that for all $x\in X_0^*$ and each family $\A=\{p_1,\ldots, p_d\}$ satisfying $(\spadesuit)$, $W_{x}^\A$ is ${\pi^*}^{\infty}$-saturated, that is,
$$W_{x}^\A=({\pi^*}^{\infty})^{-1}\Big({\pi^*}^{\infty}(W_{x}^\A)\Big).$$

  \item For each family $\A=\{p_1,\ldots, p_d\}$ satisfying $(\spadesuit)$, $M_\infty(X^*,\A)$ is ${\pi^*}^\infty$-saturated, that is,
  $$M_\infty(X^*,\A)=({\pi^*}^\infty)^{-1}\Big(M_\infty(Y^*,\A)\Big ).$$
\end{enumerate}
Thus we have the following commuting diagram:
\begin{equation*}
	\xymatrix
	{X \ar[d]_{\pi}             &
		X' \ar[l]_{\tau}
		\ar[d]^{\pi'}
		 &
		X^* \ar[l]_{\varsigma^*}
		\ar[d]^{\pi^*}
           \\
		X_\infty                 &
		Y' \ar[l]^{\theta}          &
		Y^* . \ar[l]^{\varsigma}
			}
	\end{equation*}
Since $\tau,\theta,\varsigma$ and $\varsigma^*$ are almost one to one, there are invariant residual subsets $X_1^*, X'_1,X_1,Y_1',Y_1^*$ and $(X_\infty)_1$ of $X^*, X',X,Y',Y^*$ and $X_\infty$ respectively such that in the following diagram  $\tau,\theta,\varsigma$ and $\varsigma^*$ are one to one.
\begin{equation*}
	\xymatrix
	{X_1 \ar[d]_{\pi}             &
		X_1' \ar[l]_{\tau}
		\ar[d]^{\pi'}
		 &
		X_1^* \ar[l]_{\varsigma^*}
		\ar[d]^{\pi^*}
           \\
		(X_\infty)_1                 &
		Y_1' \ar[l]^{\theta}          &
		Y_1^*  \ar[l]^{\varsigma}
			}
\end{equation*}
By Theorem \ref{thm-Fort}, the set $Y_2'$ of points of continuity of $(\pi')^{-1}: Y'\rightarrow 2^{X'}, y\mapsto (\pi')^{-1}(y)$ forms a residual subset in $Y'$, and it is clear that $Y_2'$ is invariant. Then $(\pi'\circ \varsigma^*)^{-1}(Y_2')$ is also an invariant residual subset of $X^*$, as $\pi'\circ \varsigma^*$ is semi-open.
Since $(Y^*,T)$ is an almost one to one extension of an $\infty$-step pro-nilfactor $X_\infty$, then by Lemma \ref{lem7} there is a $T$-invariant residual subset $Y_0^*$ of $Y^*$ such that for each $y^*\in X_0$, $\w_{y^*}^\A$ is a distal point.

Let
$$X^*_\bullet=X_0^*\cap X_1^*\cap (\pi'\circ \varsigma^*)^{-1}(Y_2')\cap (\pi^*)^{-1}(Y_0^*).$$
It is a $T$-invariant residual subset of $X^*$, and $X'_\bullet=\varsigma^*(X^*_\bullet)$ and $X_\bullet=\tau(X'_\bullet)$ are invariant residual subsets of $X'$ and $X$ respectively as $\varsigma^*$ and $\tau$ are almost one to one.

In the sequel we fix a point $x_0^*\in X^*_\bullet$, and let
$$x_0'=\varsigma^*(x_0^*)\in X_\bullet', \ x_0=\tau(x_0')\in X_\bullet,\  y_0^*=\pi^*(x_0^*),\  y_0'=\varsigma(y_0^*)\  \text{and}\ y_0=\theta(y_0').$$

\medskip

\noindent {\bf Step 2.} {\em For each $\ep>0$, construct a clopen partition of $X'$ and a clopen partition of $X^*$.}

\medskip

Let $\ep>0$ be a fixed positive real number. Since $X'$ is zero dimensional, there are nonempty pairwise disjoint clopen subsets $A'_1,A'_2,\ldots, A'_k$ of $X'$ such that
$$(\pi')^{-1}(y_0')\subset \bigcup_{i=1}^k A'_i,\ \ (\pi')^{-1}(y'_0)\cap A'_i\neq \emptyset,\  \text{and}\ {\rm diam} (A'_i) < \ep$$ for each $i\in \{1,2,\ldots, k\}$.
Since $\pi$ is not an almost one to one extension, each fibre of $\pi$ has at least two points and we have that $k\ge 2$ when $\ep$ is small enough.

Note that $y_0'\in Y'_2$ is a continuous point of $(\pi')^{-1}: Y'\rightarrow 2^{X'}, y\mapsto (\pi')^{-1}(y)$. As $Y'$ is also a zero dimensional space, there is a clopen neighbourhood $V$ of $y'_0$ such that
$$(\pi')^{-1}(V)\subseteq \bigcup_{i=1}^k A'_i, \ (\pi')^{-1}(y)\cap A'_i\neq \emptyset \ \text{for each}\ y\in V \ \text{and}\ i\in \{1,2,\ldots, k\}.$$
Let $A_i=A'_i\cap (\pi')^{-1}(V)$. Then $A_1,\ldots, A_k$ are nonempty pairwise disjoint clopen subsets of $X'$ with diameter less than $\ep$ and
\begin{equation}\label{h5}
  (\pi')^{-1}(V)= \bigcup_{i=1}^k A_i, \ (\pi')^{-1}(y)\cap A_i\neq \emptyset \ \text{for each}\ y\in V \ \text{and}\ i\in \{1,2,\ldots, k\}.
\end{equation}
Without loss of generality, we assume that $x'_0\in A_1$.

\medskip
\begin{figure}
	\centering
	\begin{tikzpicture}[scale = 0.8]
	\draw (0,0) rectangle (14,8);
	
	\draw[blue]( (0,1) -- (14,1);
	\draw[blue](0,4) -- (14,4);
        \draw[blue]( (0,5) -- (14,5);
	\draw[blue]( (0,6) -- (14,6);
	 \draw[blue]( (0,7) -- (14,7);
	\draw[blue]( (0,8) -- (14,8);	
		
	\draw (2,0) -- (2,8);	
	\draw (4,0) -- (4,8);	
	\draw (6,0) -- (6,8);
	\draw (8,0) -- (8,8);	
	\draw (12,0) -- (12,8);
	
	\draw[red] (1,0) -- (1,8);
	
	\draw (0,-2) -- (14,-2);
	
	\node at (0.5,7.5) {$A_1$};
	\node at (0.5,6.5) {$A_2$};
	\node at (0.5,5.5) {$A_3$};
	\node at (0.5,4.5) {$A_4$};
	\node at (0.5,2.5) {$\vdots$};
	\node at (0.5,0.5) {$A_k$};
	
	\node at (10,7.5) {$B_1$};
	\node at (10,6.5) {$B_2$};
	\node at (10,5.5) {$B_3$};
	\node at (10,4.5) {$B_4$};
	\node at (10,2.5) {$\vdots$};
	\node at (10,0.5) {$B_k$};
	
	\node at (1,8.5) {$U_0$};
	\node at (3,8.5) {$U_1$};
	\node at (5,8.5) {$U_2$};
	\node at (7,8.5) {$U_3$};
	\node at (13,8.5) {$U_{N-1}$};
	
	\node at (-0.5,6) {$X'$};
	\node at (-0.5,-2) {$Y'$};
	
	\node at (1,-1.5) {$V_0=V$};
	\node at (3,-1.5) {$V_1$};
	\node at (5,-1.5) {$V_2$};
	\node at (7,-1.5) {$V_3$};
	\node at (10,-1.5) {$\cdots$};
	\node at (13,-1.5) {$V_{N-1}$};
	
	\draw[->] (7.8,-0.5) -- (7.8, -1.2) node at (8.2, -0.8) {$\pi'$};
	
	\draw (0,-1.9)--(0,-2.1);
	\draw (2,-1.9)--(2,-2.1);
	\draw (4,-1.9)--(4,-2.1);
	\draw (6,-1.9)--(6,-2.1);
	\draw (8,-1.9)--(8,-2.1);
	\draw (12,-1.9)--(12,-2.1);
	\draw (14,-1.9)--(14,-2.1);
	\filldraw[red] (1,7.5) circle (0.1cm) node[right] {$x_0'$};
	\filldraw[red] (1,-2) circle (0.1cm) node[below] {$y_0'$};
	
	\end{tikzpicture}
	\caption{ }
\end{figure}

\medskip

Since $(Y',T)$ is minimal, there is a smallest $N\in \N$ such that $Y'=\bigcup_{j=0}^{N-1}T^{-j}V$, and $Y'\not=\bigcup_{j=0}^{N-2}T^{-j}V$.
For each $j\in \{1,2,\ldots, N\}$, we have that
$$(\pi')^{-1}(T^{-j}V)=T^{-j}((\pi')^{-1}V)=T^{-j}\big(\bigcup_{i=1}^k A_i \big )=\bigcup_{i=1}^k T^{-j}A_i,$$
and $(\pi')^{-1}(y)\cap T^{-j}A_i\neq \emptyset$ for each $y\in T^{-j}V$ and $i\in \{1,2,\ldots, k\}$.

Let $V_0=V$ and $V_n=T^{-n}V\setminus (V\cup T^{-1}V\cup\cdots \cup T^{-(n-1)}V)$ for $n\in \{1,2,\ldots, N-1\}$.
We claim that $V_n\not=\emptyset$ for $n\in \{1,2, \ldots, N-1\}$.

To show the claim we assume contrarily that there is some $M\in \{1, 2, \ldots, N-1\}$ such that $V_M=\emptyset$. Then $T^{-M}V\subseteq \bigcup_{j=0}^{M-1}T^{-j}V$. And
$$T^{-(M+1)}V=T^{-1}T^{-M}V\subseteq T^{-1}\bigcup_{j=0}^{M-1}T^{-j}V=\big(\bigcup_{j=1}^{M-1}T^{-j}V\big)\cup T^{-M}V\subseteq \bigcup_{j=0}^{M-1}T^{-j}V.$$
Inductively, we have $T^{-(M+i)}V \subseteq \bigcup_{j=0}^{M-1}T^{-j}V$ for all $i\in \N$. Thus $Y'=\bigcup_{j=0}^{N-1}T^{-j}V\subseteq  \bigcup_{j=0}^{M-1}T^{-j}V$, which contradicts with the fact $N$ is the smallest number such that $Y'=\bigcup_{j=0}^{N-1}T^{-j}V$.
This ends the proof of the claim.

Then $\{V_0,V_1,\ldots, V_{N-1}\}$ is a clopen cover of $Y'$ and $V_0,V_1,\ldots, V_{N-1}$ are nonempty and pairwise disjoint.
For each $j\in \{0,1,\ldots, N-1\}$, set $U_j=(\pi')^{-1}(V_j)$. Then $\{U_0,U_1,\ldots, U_{N-1}\}$ is a pairwise disjoint clopen cover of $X'$. As $V_j\subseteq T^{-j}V$, we have
$$(\pi')^{-1}(V_j)= \bigcup_{i=1}^k T^{-j}A_i\cap U_j, \ (\pi')^{-1}(y)\cap T^{-j}A_i\cap U_j\neq \emptyset$$
 for each $ y\in V_j \ \text{and}\ i\in \{1,2,\ldots, k\}.$
Note that $T^{-j}A_1\cap U_j, \ldots, T^{-j}A_k\cap U_j$ are pairwise disjoint nonempty clopen subsets of $X'$. Now for each $i\in \{1,2,\ldots, k\}$ let
$$B_i=\bigcup_{j=0}^{N-1}T^{-j}A_i\cap U_j .$$
Then $B_1, B_2,\ldots, B_k$ are pairwise disjoint clopen subsets of $X'$, $X'=\bigcup_{i=1}^k B_i$ and for each $y\in Y'$, $(\pi')^{-1}(y)\cap B_i\neq \emptyset$ for each $i\in \{1,2,\ldots, k\}$. And note that $B_i\cap (\pi')^{-1}(V)=A_i$ for each $i\in \{1,2,\ldots, k\}$.

\medskip

Let $V^*=\varsigma^{-1}(V)$, $A_i^*=(\varsigma^*)^{-1}(A_i)$, $B_i^*=(\varsigma^*)^{-1}(B_i)$ for each $i\in \{1,2,\ldots, k\}$.
Then
$$(\pi^*)^{-1}(V^*)= \bigcup_{i=1}^k A^*_i\ \text{and} \ X^*=\bigcup_{i=1}^k B^*_i,$$
and $B^*_1, B^*_2,\ldots, B^*_k$ are pairwise disjoint nonempty clopen subsets of $X^*$. It is easy to check that
$$A^*_1=B_1^*\cap (\pi^*)^{-1}(V^*),\ A^*_2=B_2^*\cap (\pi^*)^{-1}(V^*),  \ldots, A^*_k=B_k^*\cap (\pi^*)^{-1}(V^*)$$ are pairwise disjoint nonempty clopen subsets of $X^*$.
Since $x_0'\in A_1$, we have that $x_0^*\in A_1^*$.

\medskip

\begin{figure}
	\centering
	\begin{tikzpicture}	
	\draw (0,0) rectangle	(10,6);
	\draw[blue] (0,1) -- (10,1);	
	\draw[blue] (0,3) -- (10,3);	
	\draw[blue] (0,4) -- (10,4);	
	\draw[blue] (0,5) -- (10,5);	
	
	\draw[red] (1,0) -- (1,6);	
	\draw (2,0) -- (2,6);
	\draw[red] (4,0) -- (4,6);	
	
	\draw (0,-1.3)-- (10,-1.3);	
	
	\node at (-0.5, 4) {$X^{*}$};
	\node at (-0.5, -1.2) {$Y^{*}$};
	
	\node[red] at (0.5,0.5)  {$A_{k}^{*}$};
	\node[red] at (0.5,3.5)  {$A_{3}^{*}$};
	\node[red] at (0.5,4.5)  {$A_{2}^{*}$};
	\node[red] at (0.5,5.5)  {$A_{1}^{*}$};
	\node[blue] at (7,0.5)  {$B_{k}^{*}$};
	\node[blue] at (7,2)  {$\vdots$};
	\node[blue] at (7,3.5)  {$B_{3}^{*}$};
	\node[blue] at (7,4.5)  {$B_{2}^{*}$};
	\node[blue] at (7,5.5)  {$B_{1}^{*}$};
	
	\filldraw[red] (1,5.5) circle (0.1cm) node [right] {$x_{0}^{*}$};
	\filldraw[red] (1,-1.3) circle (0.1cm) node [above=0.2cm, right] {$y_{0}^{*}$};
	\filldraw[red] (4,-1.3) circle (0.1cm) node [above=0.2cm, right] {$y^{*}$};
	
	\draw[red, ->] (4.5, 2)--(4,2.3) node at (5.5,2) {$(\pi^{*})^{-1}(y^{*})$};
	\draw (0,-1.2)--(0,-1.4);
         \draw (2,-1.2)--(2,-1.4);
         \draw (10,-1.2)--(10,-1.4);

         \draw[->] (5, -0.4)--(5, -1) node at (5.4, -0.7) {$\pi^{*}$};
         \node at (1.3, -1.6) {$V^{*}$};
	\end{tikzpicture}
\caption{ }
\end{figure}
\medskip

\noindent {\bf Claim} \ {\em For each $i\in \{1,2,\ldots, k\}$, $(\pi^*)^{-1}(y^*)\cap A^*_i\neq \emptyset$ for each $y^*\in V^*$, and $(\pi^*)^{-1}(y^*)\cap B^*_i\neq \emptyset$ for each $y^*\in Y^*$.
}

\begin{proof}[Proof of Claim]
To prove the claim, we need to use the construction of O-diagram.
\begin{equation*}
	\xymatrix
	{X' \ar[d]_{\pi'}
		 &
		X^* \ar[l]_{\varsigma^*}
		\ar[d]^{\pi^*}
           \\
		Y'          &
		Y^* \ar[l]^{\varsigma}
			}
	\end{equation*}
First recall how to construct the O-diagram.
The map $(\pi')^{-1}: Y'\rightarrow 2^{X'}, y\mapsto (\pi')^{-1}(y)$ is a u.s.c. map, and the set $Y_2'$ of
continuous points of $(\pi')^{-1}$ is a dense $G_\d$ subset of $Y'$.
Let $Y^*=\overline{\{(\pi')^{-1}(y): y\in Y_2'\}}$,
where the closure is taken in $2^{X'}$. For each $A\in Y^*$, there is some $y\in Y'$ such that $A\subseteq (\pi')^{-1}(y)$, and hence $A\mapsto y$
defines the map $\varsigma: Y^*\rightarrow Y'$.
It can be proved that $X^*=\{(x,\tilde y)\in X \times
Y^* :  x \in \tilde y\}$. $\varsigma^*$ and
$\pi^*$ are the restrictions to $X^*$ of the projections of
$X'\times Y^*$ onto $X'$ and $Y^*$ respectively.

\medskip

We need to show that for each $i\in \{1,2,\ldots, k\}$, $(\pi^*)^{-1}(y^*)\cap A^*_i\neq \emptyset$ for each $y^*\in V^*$. The fact that $(\pi^*)^{-1}(y^*)\cap B^*_i\neq \emptyset$ for each $y^*\in Y^*$ will follow readily.

Let $y^*\in V^*$. Then $y'=\varsigma(y^*)\in V$. Since $(Y',T)$ is minimal and $y_0'\in Y_2'$, there is some sequence $n_t\nearrow\infty, t\to\infty$ such that $T^{n_t}(\pi')^{-1} (y_0')=(\pi')^{-1} (T^{n_t}y_0')\to y^*, t\to\infty$, in the space $2^{X'}$. Via the map $\varsigma$, we have $T^{n_t} y_0'\to y', t\to\infty$ in $Y'$.
There is some $t_0\in \N$ such that $T^{n_t}y_0'\in V$ for all $t\ge t_0$. By \eqref{h5}, when $t\ge t_0$, we have
$$(\pi')^{-1} (T^{n_t}y_0')\subseteq \bigcup_{i=1}^kA_i\ \text{and}\ (\pi')^{-1}(T^{n_t}y_0')\cap A_i\neq \emptyset$$ for each $i\in \{1,2,\ldots, k\}$.
As $(\pi')^{-1} (T^{n_t}y_0')\to y^*, t\to\infty$ in $2^{X'}$, it follows that  $y^*\subseteq \bigcup_{i=1}^kA_i$ and $y^*\cap A_i\neq \emptyset$ for each $i\in \{1,2,\ldots, k\}$. For each $i\in \{1,2,\ldots, k\}$, let $x_i\in y^*\cap A_i$. Then $(x_i,y^*)\in (\pi^*)^{-1}(y^*) \subseteq X^*$ and $\varsigma^*(x_i,y^*)=x_i\in A_i$. That is $(x_i,y^*)\in A_i^*\cap (\pi^*)^{-1}(y^*)$. Thus we have proved that for each $i\in \{1,2,\ldots, k\}$, $(\pi^*)^{-1}(y^*)\cap A^*_i\neq \emptyset$ for each $y^*\in V^*$. This ends the proof of {Claim}.
\end{proof}


Let $E$ be a space.
A point of $(E^d)^{\Z}$ is denoted by
$$\pmb{\xi}=(\pmb{\xi}_n)_{n\in {\Z}}=\Big((\xi^{(1)}_n, \xi^{(2)}_n,\cdots, \xi^{(d)}_n) \Big)_{n\in \Z}.$$
Let $\sigma: (E^d)^{{\Z}}\rightarrow (E^d)^{{\Z}}$
be the shift map, i.e.,  for all $\pmb{\xi}=(\pmb{\xi}_n)_{n\in {\Z}}\in (Z^d)^{{\Z}}$
$$(\sigma \pmb{\xi})_n=\pmb{\xi}_{n+1}, \ \forall n\in {\Z}. $$
In the following step, we will take $E=\{1,2,\ldots,k\}$, $k\ge 2$.

\medskip

\noindent {\bf Step 3.} {\em For given $\ep>0$, let $\{B_1^*,\ldots, B_k^*\}$ be the partition of $X^*$ in Step 2 with respect to $\ep$. Define $\pmb{\xi}=(\pmb{\xi}_n)_{n\in {\Z}}=\Big((\xi^{(1)}_n, \xi^{(2)}_n,\cdots, \xi^{(d)}_n) \Big)_{n\in \Z} \in \Sigma_k=\Big(\{1,2,\ldots, k\}^d\Big)^\Z$ as follows:
\begin{equation}\label{h1}
  T^{p_i(n)}x_0^*\in B_j^* \Longleftrightarrow \xi_n^{(i)}=j,\ \text{where}\ i\in \{1,2,\ldots, d\}, \ j\in \{1,2,\ldots, k\},\ n\in \Z.
\end{equation}
Then
\begin{itemize}
  \item $\pmb \xi$ is a transitive point of $(\Sigma_k,\sigma)$;
  \item for any minimal t.d.s. $(W,H)$ and $w\in W$, we have that
\begin{equation}\label{h2}
  \overline{\O}\Big((w,\pmb{\xi}), H \times \sigma\Big)=W \times \Sigma_k.
\end{equation}
\end{itemize}

In particular, in $(((Y^*)^d)^\Z\times \Sigma_k, \sigma\times \sigma)$, we have that
\begin{equation}\label{}
  \overline{\O}\Big((\w_{y_0^*}^\A,\pmb\xi), \sigma\times \sigma\Big)=W_{y_0^*}^\A\times \Sigma_k.
\end{equation}

}


\begin{proof} 
First we show that $\pmb \xi$ is a transitive point of $(\Sigma_k,\sigma)$. To see this, we show that for each $\pmb{\w}\in \Sigma_k$ and each $N\in \N$, there is some $n\in \N$ such that $\rho(\sigma^n\pmb\xi,\pmb\w)<\frac{1}{N+1}$.
where the metric $\rho$ on $\Sigma_k$ be defined as follows: for ${\bf x}=({\bf x}_n)_{n\in \Z}, {\bf y}=({\bf y}_n)_{n\in \Z}$,
$$\rho({\bf x}, {\bf y})=0\ \text{if}\ {\bf x}={\bf y};\ \  \rho({\bf x}, {\bf y})=\frac{1}{m+1}, \ \text{if}\ {\bf x}\not={\bf y}\ \text{and}\ m=\min\{|j|: \exists 1\le i\le d\ s.t.\  x^{(i)}_j\neq y^{(i)}_j \}.$$
By the definition of $\rho$, we need to show that there is some $n\in \N$ such that
$$\pmb\xi_{n+j}=\pmb\w_j, \ \forall j\in \{-N,-N+1,\ldots, N\}.$$
By Claim in Step 2, for each $i\in \{1,2,\ldots, k\}$, $\pi^*(B_i^*)=Y^*$. Thus $((\pi^*)^d)^{2N+1}\big(\prod_{j=-N}^N B_{\pmb\w_j}^*\big)=((Y^*)^d)^{2N+1}$, where $B_{\pmb\w_j}=B^*_{\w^{(1)}_j}\times B^*_{\w^{(2)}_j}\times \cdots\times B^*_{\w^{(d)}_j}$. It follows that
$$\Big(\prod_{i=-\infty}^{-N-1}(X^*)^d \times \prod_{j=-N}^N B_{\pmb\w_j}^*\times \prod_{i=N+1}^\infty (X^*)^d\Big)\cap  ({\pi^*}^{\infty})^{-1}\Big(W_{y_0^*}^\A\Big)\neq \emptyset.$$
By the choice of $x_0^*$, $W_{x_0^*}^\A$ is ${\pi^*}^{\infty}$-saturated, that is,
$$W_{x_0^*}^\A=({\pi^*}^{\infty})^{-1}\Big({\pi^*}^{\infty}(W_{x_0^*}^\A)\Big)=
({\pi^*}^{\infty})^{-1}\Big(W_{y_0^*}^\A\Big).$$
Thus $$\Big(\prod_{i=-\infty}^{-N-1}(X^*)^d \times \prod_{j=-N}^N B_{\pmb\w_j}^*\times \prod_{i=N+1}^\infty (X^*)^d\Big)\cap  W_{x_0^*}^\A\neq \emptyset.$$
That is, $\Big(\prod_{i=-\infty}^{-N-1}(X^*)^d \times \prod_{j=-N}^N B_{\pmb\w_j}^*\times \prod_{i=N+1}^\infty (X^*)^d\Big) \cap  W_{x_0^*}^\A$ is a nonempty open subset of $W_{x_0^*}^\A$. There is some $n\in \N$ such that
$$\sigma^n \w_{x_0^*}^\A \in \prod_{i=-\infty}^{-N-1}(X^*)^d \times \prod_{j=-N}^N B_{\pmb\w_j}^*\times \prod_{i=N+1}^\infty (X^*)^d,$$
which implies that
$$T^{p_i(n+j)}x_0^*\in B^*_{\w_j^{(i)}}, \ \text{for all}\ i\in \{1,2,\ldots, d\}, j\in \{-N,-N+1,\ldots, N\}. $$
By \eqref{h1}, this means that
$$\pmb\xi_{n+j}=\pmb\w_j, \  \text{for all}\ j\in \{-N,-N+1,\ldots, N\}.$$
Hence $\pmb\xi$  is a transitive point of $(\Sigma_k,\sigma)$.

Now let $(W,H)$ be a minimal t.d.s. and $w\in W$. Since $\pmb\xi$ is a transitive point of $(\Sigma_k,\sigma)$ and $W=\overline{\O}(w,H)$, $\overline{\O}\Big((w,\xi), H\times \sigma\Big)$ is a joining of t.d.s. $(W, H)$ and $(\Sigma_k,\sigma)$. Since $(W,H)$ is a minimal system, by Furstenberg disjointness theorem, 
$(\Sigma_k,\sigma)$ is disjoint from $(W, H)$. It follows that
$$\overline{\O}\Big((w,\pmb{\xi}), H \times \sigma\Big)=W \times \Sigma_k.$$
So we have \eqref{h2}. 
\end{proof}

\medskip

\noindent {\bf Step 4.} {\em
Let $(Z, S)$ be a minimal t.d.s. and $z\in Z$. We show that for each $\ep>0$ and $x_0'\in X_\bullet'$,
$$N_{z,x_0',\ep}=\{n\in \Z: \rho_Z(S^n z,z)<\ep \ \text{and}\ \rho_{X'}(T^{p_1(n)}x_0', x_0')<\ep,\ \ldots,\ \rho_{X'}(T^{p_d(n)}x_0', x_0')<\ep \}$$
is a piecewise syndetic subset.}

\begin{proof}
For each given $\ep>0$, we have the constructions and results in Step 2 and Step 3.
By Step 3, we have that
$$\overline{\O}\Big((\w_{y_0^*}^\A,\pmb\xi), \sigma\times \sigma\Big)=W_{y_0^*}^\A\times \Sigma_k.$$
Since $y_0^*=\pi^*(x^*_0)\in Y_0^*$, $\w_{y_0^*}^\A$ is a distal point and it is product recurrent. Thus $(z, \w_{y_0^*}^\A)$ is a minimal point of $(Z\times W_{y_0^*}^\A, S\times \sigma)$. By Step 3 again, we have that
$$\overline{\O}\Big((z, \w_{y_0^*}^\A,\pmb\xi),(S\times \sigma)\times \sigma\Big)=\overline{\O}\big((z,\w_{y_0^*}^\A), S\times \sigma\big)\times \Sigma_k.$$
It follows that
$$(z, \w_{y_0^*}^\A, \pmb\xi)\in \overline{\O}\Big((z, \w_{y_0^*}^\A, \pmb\xi), (S\times \sigma)\times \sigma\Big),$$
and by Theorem \ref{prop-M} $(z, \w_{y_0^*}^\A, \pmb\xi)$ is a piecewise syndetic recurrent point since minimal points are dense in $\overline{\O}\Big((z, \w_{y_0^*}^\A,\pmb\xi),(S\times \sigma)\times \sigma\Big)=\overline{\O}\big((z,\w_{y_0^*}^\A), S\times \sigma\big)\times \Sigma_k.$

Recall that $y_0^*\in V^*$ and $T^{p_i(0)}x_0^*=x_0^*\in A_1^*\subseteq B^*_1$, for all $i\in \{1,2,\ldots,d\}$. Hence $\pmb\xi_0=(1,1,\ldots,1)=1^{\otimes d}$ and $[1^{\otimes d}]=\{\pmb\w\in \Sigma_k: \pmb\w_0=1^{\otimes d}\}$ is a neighbourhood of $\pmb \xi$. It follows that $ B(z,\ep)\times [(V^*)^{\otimes d}]\times [1^{\otimes d}]$ is a neighbourhood of $(z, \w_{y_0^*}^\A, \pmb\xi)$, where $[(V^*)^{\otimes d}]= \{{\bf y}\in ((Y^*)^d)^\Z: {\bf y}_0 \in V^*\times V^*\times \cdots\times V^* \}$. Thus
$$N_1\triangleq \{n\in \Z:(S^nz, \sigma^n\w_{y_0^*}^\A, \sigma^n\pmb \xi)\in B(z,\ep)\times [(V^*)^{\otimes d}]\times [1^{\otimes d}]\} $$
is piecewise syndetic.

Let $n\in N_1$.
By $\sigma^{n}\w_{y_0^*}^\A\in [(V^*)^{\otimes d}]$, we have $T^{p_i(n)}y_0^*\in V^*$ for all $i\in \{1,2,\ldots,d\}$. Hence
\begin{equation}\label{h3}
  T^{p_i(n)}x_0^*\in (\pi^*)^{-1}(V^*)=\bigcup_{j=1}^k A_j^*\quad \text{for all}\ i\in \{1,2,\ldots, d\}.
\end{equation}
Since $\sigma^n\pmb\xi\in [1^{\otimes d}]$, we have $\pmb\xi_n=1^{\otimes d}$. By \eqref{h1},
\begin{equation}\label{h4}
  T^{p_i(n)}x_0^*\in B_1^*\quad \text{for all}\ i\in \{1,2,\ldots, d\}.
\end{equation}
By \eqref{h3} and \eqref{h4}, we have $ T^{p_i(n)}x_0^*\in A_1^*=B_1^*\cap (\pi^*)^{-1}(V^*)$ for all $i\in \{1,2,\ldots, d\}$. So $T^{p_i(n)}x_0'\in A_1$ for all $i\in \{1,2,\ldots, d\}$. As $x_0'\in A_1$ and ${\rm diam}( A_1)<\ep$, we have that $\rho_{X'}(T^{p_i(n)}x_0',x_0')<\ep$ for all $i\in \{1,2,\ldots, d\}$. Combining with $S^nz\in B(z,\ep)$, we have that
$$ \rho_Z(S^n z,z)<\ep \ \text{and}\ \rho_{X'}(T^{p_1(n)}x_0', x_0')<\ep,\ \ldots,\ \rho_{X'}(T^{p_d(n)}x_0', x_0')<\ep.$$
So
$$N_1\subseteq N_{z,x_0',\ep},$$
and $N_{z,x_0',\ep}$ is piecewise syndetic.
The proof is complete.
\end{proof}

\medskip

\noindent {\bf Step 5.} {\em
Let $(Z, S)$ be a minimal t.d.s. and $z\in Z$. We show that for each $\ep>0$ and $x_0\in X_\bullet$,
$$N_{z,x_0,\ep}=\{n\in \Z: \rho_Z(S^n z,z)<\ep \ \text{and}\ \rho_{X}(T^{p_1(n)}x_0, x_0)<\ep,\ \ldots,\ \rho_{X}(T^{p_d(n)}x_0, x_0)<\ep \}$$
is a piecewise syndetic subset.}

\medskip

\begin{proof}
For each $x_0\in X_\bullet$, there is some $x_0'\in X_\bullet'$ such that $x_0=\tau(x_0')$. Since $\tau^{-1}(B(x_0,\ep))$ is an open neighbourhood of $x_0'$, choose some $\ep'>0$ such that $B(x_0',\ep')\subseteq \tau^{-1}(B(x_0,\ep))$.
Then $$N_{z,x_0',\ep'}\subseteq N_{z,x_0,\ep}.$$
By Step 4, $N_{z,x_0',\ep'}$ is piecewise syndetic, and it follows that $N_{z,x_0,\ep}$ is piecewise syndetic.
\end{proof}
The whole proof is complete.
\end{proof}

\subsection{Proof of Theorem \ref{thm-main1} for the general case}\label{subsect-7.3}
\
\medskip

In the previous subsection we show that our main theorem holds if the family $\A$ of nonlinear integral polynomials satisfies the condition $(\spadesuit).$
For the general case, the proof is basic the same to the one in Subsection \ref{subsec7.3} and we need to modify Step 2 and Step 4 in the previous subsection.

\medskip

Without loss of generality, we assume that nonlinear integral polynomials $p_1(n)$, $p_2(n)$, $\ldots$, $p_d(n)$ are essentially distinct.
Let  $\A=\{p_1(n), \ldots, p_{d}(n)\}$. We have showed that our main theorem holds if $\A$ satisfies the condition $(\spadesuit)$ in Subsection \ref{subsec7.3}. Generally,
by Lemma \ref{ww=dem}, there are nonlinear integral polynomials $\A'=\{q_1,\ldots, q_r\}$ satisfying the condition ($\spadesuit$) and $\{p_{1},\ldots, p_d\}\subseteq \{q_i^{[j]}=q_i(n+j)-q_i(j): 1\le i\le r, j\in \Z \}.$ Thus there is some $l\in \N$ such that
$$\{p_{1},\ldots, p_d\}\subseteq \{q_i^{[j]}=q_i(n+j)-q_i(j): 1\le i\le r, -l\le j\le l\}.$$

By Lemma \ref{lem7}, it remains to show the case when $(X,T)$ is not an almost one to one extension of $X_\infty$. Without loss of generality we may assume that the maximal equicontinuous factor of $X$ is infinite, since if it is not the case
we may replace $X$ by its product with an irrational rotation.
We divide the proof into the following steps.


\medskip

\noindent {\bf Step 1.} {\em Notations and preparation.}

\medskip

Let $(X,T)$ be a minimal t.d.s., and $\pi:X\rightarrow X_\infty$ be the factor map from $X$ to its maximal $\infty$-step pro-nilfactor $X_\infty$. By the Step 1 of Subsection \ref{subsec7.3}, we have the following commuting diagram:
\begin{equation*}
	\xymatrix
	{X \ar[d]_{\pi}             &
		X' \ar[l]_{\tau}
		\ar[d]^{\pi'}
		 &
		X^* \ar[l]_{\varsigma^*}
		\ar[d]^{\pi^*}
           \\
		X_\infty                 &
		Y' \ar[l]^{\theta}          &
		Y^* . \ar[l]^{\varsigma}
			}
	\end{equation*}
Let $X^*_\bullet$, $X'_\bullet=\varsigma^*(X^*_\bullet)$ and $X_\bullet=\tau(X'_\bullet)$ be defined as in Subsection \ref{subsec7.3}.

In the sequel we fix a point $x_0^*\in X^*_\bullet$, and let
$$x_0'=\varsigma^*(x_0^*)\in X_\bullet', \ x_0=\tau(x_0')\in X_\bullet,\  y_0^*=\pi^*(x_0^*),\  y_0'=\varsigma(y_0^*)\  \text{and}\ y_0=\theta(y_0').$$

\medskip

\noindent {\bf Step 2.} {\em For each $\ep>0$, construct a clopen partition of $X'$ and a clopen partition of $X^*$.}

\medskip

Let $\ep>0$ be a fixed real number. Let
$$L=\max\{|q_i(j)|: 1\le i\le r, -l\le j\le l\}.$$
As $T$ is a homeomorphism, there is some $\d\in (0,\ep)$ such that whenever $\rho_{X'}(x,x')<\d$, one has that $\rho_{X'}(T^nx,T^nx')<\ep$ for all $n\in \{-L,-L+1,\ldots, L\}$. Fix such $\d>0$, and choose a $\d'\in (0,\d)$ such that whenever $\rho_{X'}(x,x')<\d'$, one has that $\rho_{X'}(T^nx,T^nx')<\d$ for all $n\in \{-L,-L+1,\ldots, L\}$.

Since $X'$ is zero dimensional, there are nonempty pairwise disjoint clopen subsets $A'_1,A'_2,\ldots, A'_k$ of $X'$ such that
$$(\pi')^{-1}(y_0')\subset \bigcup_{i=1}^k A'_i,\ \ (\pi')^{-1}(y'_0)\cap A'_i\neq \emptyset,\  \text{and}\ {\rm diam} (A'_i) < \d' \ \text{for each}\ i\in \{1,2,\ldots, k\}.$$
Since $\pi$ is not an almost one to one extension, each fibre of $\pi$ has at least two points and we have that $k\ge 2$ when $\d'$ is small enough.

Note that $y_0'\in Y'_2$ is a continuous point of $(\pi')^{-1}: Y'\rightarrow 2^{X'}, y\mapsto (\pi')^{-1}(y)$. As $Y'$ is also a zero dimensional space, there is a clopen neighbourhood $V$ of $y'_0$ such that
\begin{itemize}
  \item $(\pi')^{-1}(V)\subseteq \bigcup_{i=1}^k A'_i, \ (\pi')^{-1}(y)\cap A'_i\neq \emptyset \ \text{for each}\ y\in V \ \text{and}\ i\in \{1,2,\ldots, k\},$
  \item $\{T^jV\}_{j=-L}^L$ are disjoint.
\end{itemize}

Let $A_i=A'_i\cap (\pi')^{-1}(V)$. Then $A_1,\ldots, A_k$ are nonempty pairwise disjoint clopen subsets of $X'$ with diameter less than $\d'$ and
\begin{equation}\label{e1}
  (\pi')^{-1}(V)= \bigcup_{i=1}^k A_i, \ (\pi')^{-1}(y)\cap A_i\neq \emptyset \ \text{for each}\ y\in V \ \text{and}\ i\in \{1,2,\ldots, k\}.
\end{equation}
Without loss of generality, we assume that $x'_0\in A_1$. By the definition of $\d'$, we have that for each $j\in \{-L,-L+1,\ldots, L\}$, $T^jA_1,\ldots, T^jA_k$ are nonempty pairwise disjoint clopen subsets of $X'$ with diameter less than $\d$ and
\begin{equation}\label{e7}
  (\pi')^{-1}\left(T^jV\right)= \bigcup_{i=1}^k T^j A_i, \ (\pi')^{-1}(y)\cap T^jA_i\neq \emptyset \ \text{for each}\ y\in T^jV \ \text{and}\ i\in \{1,2,\ldots, k\}.
\end{equation}

Next repeat the construction in Step 2 of Subsection \ref{subsec7.3} to get clopen partition $\{B_1, B_2, \ldots, B_k\}$ of $X'$. Since $\{T^jV\}_{j=-L}^L$ are disjoint clopen subsets of $Y'$, there will be some additional properties.

Since $(Y',T)$ is minimal, there is a smallest $N\in \N$ such that $Y'=\bigcup_{j=-N}^{N}T^{-j}V$. Since $\{T^jV\}_{j=-L}^L$ are disjoint clopen subsets of $Y'$, we have $N\ge L$.
Let $V_0=V$ and $V_n=T^{n}V\setminus (V\cup T^{1}V\cup T^{-1}V\cup \cdots \cup T^{-(n-1)}V)$ and $V_{-n}=T^{-n}V\setminus (V\cup T^{1}V\cup T^{-1}V\cup \cdots \cup T^{-(n-1)}V\cup T^nV)$ for $n\in \{0,1,\ldots, N\}$.
We have that $V_n\not=\emptyset$ for $n\in \{-N,-N+1,\ldots, N\}$ and $V_j=T^{j}V$ for all $j\in \{-L,-L+1,\ldots, L\}$.

\medskip

\begin{figure}
	\centering
	\begin{tikzpicture}	
	\draw (0,0) rectangle	(12,7);
		
	\draw[blue] (0,1)--(12,1);
	\draw[blue] (0,4)--(12,4);
	\draw[blue] (0,5)--(12,5);
	\draw[blue] (0,6)--(12,6);
	
	\draw (2,0)--(2,7);	
	\draw (2.5,0)--(2.5,7);
	\draw (3.5,0)--(3.5,7);
	\draw (4,0)--(4,7);
	\draw[red] (4.8,0)--(4.8,7);
	\draw (5.6,0)--(5.6,7);
	\draw (6.6,0)--(6.6,7);
	\draw (7.6,0)--(7.6,7);
	\draw (9,0)--(9,7);
	\draw (10,0)--(10,7);
	
	\node[red] at (4.4,6.5) {$A_1$};
	\node[red] at (4.4,5.5) {$A_2$};
	\node[red] at (4.4,4.5) {$A_3$};
	\node[red] at (4.4,0.5) {$A_k$};
	\node[blue] at (8.5,6.5) {$B_1$};
	\node[blue] at (8.5,5.5) {$B_2$};
	\node[blue] at (8.5,4.5) {$B_3$};
	\node[blue] at (8.5,0.5) {$B_k$};
	\node[blue] at (8.5,2.5) {$\vdots$};
	
	\filldraw[red]  (4.8, 6.5) circle (0.05cm) node[right] {$x_0'$};
	
	\draw(0,-2)--(12,-2);
	\draw(0,-1.9)--(0,-2.1);
	\draw(0.4,-1.9)--(0.4,-2.1);
	\draw(1,-1.9)--(1,-2.1);
	\draw(2,-1.9)--(2,-2.1);
	\draw(2.5,-1.9)--(2.5,-2.1);
	\draw(3.5,-1.9)--(3.5,-2.1);
	\draw(4,-1.9)--(4,-2.1);
	\draw[red](4.8,-1.9)--(4.8,-2.1) node [below] {$y_0'$};
	\draw(5.6,-1.9)--(5.6,-2.1);
	\draw(6.6,-1.9)--(6.6,-2.1);
	\draw(7.6,-1.9)--(7.6,-2.1);
	\draw(9,-1.9)--(9,-2.1);
	\draw(10,-1.9)--(10,-2.1);
	\draw(10.6,-1.9)--(10.6,-2.1);
	\draw(11.5,-1.9)--(11.5,-2.1);
	\draw(12,-1.9)--(12,-2.1);
	
	\node[rotate=60] at (0.2, -1.3) {\small $V_{-N}$};
	\node[rotate=45] at (1, -1.3) {\small $V_{-L-1}$};
	\node at (1.6, -2.3) {\small $V_{-L}=T^{-L}V$};
	\node[rotate=60] at (3.6, -1) {\small $V_{-1}=T^{-1}V$};
	\node at (4.5, -1.7) {\small $V_{0}=V$};
	\node[rotate=60] at (6, -1.2) {\small $V_{1}=TV$};
	\node[rotate=60] at (10, -1.2) {\small $V_{L}=T
	^{L}V$};
	\node at (10.4,-2.3) {\small $V_{L+1}$};
	\node at (11.2,-2.3) {\small $\cdots$};
	\node at (11.8,-2.3) {\small $V_{N}$};
	
	\draw[->] (7.5,-0.5)--(7.5, -1.2) node at (7.8,-0.8) {$\pi'$};
	\end{tikzpicture}
\caption{ }
\end{figure}

\medskip

Next we repeat the construction in Step 2 of Subsection \ref{subsec7.3} to get clopen subsets of $\{B_1, B_2,\ldots, B_k\}$ of $X'$
such that
\begin{itemize}
  \item $X'=\bigcup_{i=1}^k B_i$;
  \item for each $y\in Y'$, $(\pi')^{-1}(y)\cap B_i\neq \emptyset$ for each $i\in \{1,2,\ldots, k\}$;
  \item $B_i\cap (\pi')^{-1}(T^jV)=T^j A_i $ has diameter less than $\d$ for each $i\in \{1,2,\ldots, k\}$ and each $j\in \{-L,-L+1,\ldots, L\}$. In particular, $B_i\cap (\pi')^{-1}(V)=A_i$ for each $i\in \{1,2,\ldots, k\}$.
\end{itemize}

\medskip

Let $V^*=\varsigma^{-1}(V)$, $A_i^*=(\varsigma^*)^{-1}(A_i)$, $B_i^*=(\varsigma^*)^{-1}(B_i)$ for each $i\in \{1,2,\ldots, k\}$.
Then
$$(\pi^*)^{-1}(V^*)= \bigcup_{i=1}^k A^*_i, \quad X^*=\bigcup_{i=1}^k B^*_i,$$
and $B^*_1, B^*_2,\ldots, B^*_k$ are pairwise disjoint nonempty clopen subsets of $X^*$. Since $x_0'\in A_1$, we have that $x_0^*\in A_1^*$.
Also we have that
\begin{itemize}
    \item for each $i\in \{1,2,\ldots, k\}$, $(\pi^*)^{-1}(y^*)\cap A^*_i\neq \emptyset$ for each $y^*\in V^*$;
  \item $(\pi^*)^{-1}(y^*)\cap B^*_i\neq \emptyset$ for each $y^*\in Y^*$;
  \item $T^j A^*_1=B_1^*\cap (\pi^*)^{-1}(T^j V^*),\ T^j A^*_2=B_2^*\cap (\pi^*)^{-1}(T^jV^*),  \ldots, T^jA^*_k=B_k^*\cap (\pi^*)^{-1}(T^jV^*)$ are pairwise disjoint nonempty clopen subsets of $X^*$ for all $j\in \{-L,-L+1,\ldots, L\}$. In particular,  $A^*_1=B_1^*\cap (\pi^*)^{-1}(V^*),\ A^*_2=B_2^*\cap (\pi^*)^{-1}(V^*),  \ldots, A^*_k=B_k^*\cap (\pi^*)^{-1}(V^*)$ are pairwise disjoint nonempty clopen subsets of $X^*$.
\end{itemize}

\medskip

\medskip

\noindent {\bf Step 3.} {\em For given $\ep>0$, let $\{B_1^*,\ldots, B_k^*\}$ be the partition of $X^*$ in Step 2 with respect to $\ep$. Define $\pmb{\xi}=(\pmb{\xi}_n)_{n\in {\Z}}=\Big((\xi^{(1)}_n, \xi^{(2)}_n,\cdots, \xi^{(r)}_n) \Big)_{n\in \Z} \in \Sigma_k=\Big(\{1,2,\ldots, k\}^r\Big)^\Z$ as follows:
\begin{equation}\label{e2}
  T^{q_i(n)}x_0^*\in B_j^* \Longleftrightarrow \xi_n^{(i)}=j,\ \text{where}\ i\in \{1,2,\ldots, r\}, j\in \{1,2,\ldots, k\}.
\end{equation}
Then
\begin{itemize}
  \item $\pmb \xi$ is a transitive point of $(\Sigma_k,\sigma)$;
  \item for any minimal t.d.s. $(W,H)$ and $w\in W$, we have that
\begin{equation}\label{e3}
  \overline{\O}\Big((w,\pmb{\xi}), H \times \sigma\Big)=W \times \Sigma_k.
\end{equation}
\end{itemize}

In particular, in $(((Y^*)^r)^\Z\times \Sigma_k, \sigma\times \sigma)$, we have that
\begin{equation}\label{e4}
  \overline{\O}\Big((\w_{y_0^*}^{\A'},\pmb\xi), \sigma\times \sigma\Big)=W_{y_0^*}^{\A'}\times \Sigma_k.
\end{equation}

}

Proof of Step 3 is the same to the one in Subsection \ref{subsec7.3}.

\medskip

\noindent {\bf Step 4.} {\em
Let $(Z, S)$ be a minimal t.d.s. and $z\in Z$. We show that for each $\ep>0$ and $x_0'\in X_0'$,
$$N_{z,x_0',\ep}=\{n\in \Z: \rho_Z(S^n z,z)<\ep \ \text{and}\ \rho_{X'}(T^{p_1(n)}x_0', x_0')<\ep,\ \ldots,\ \rho_{X'}(T^{p_d(n)}x_0', x_0')<\ep\}$$
is piecewise syndetic.}

\begin{proof}
For each given $\ep>0$, we have the constructions and results in Step 2 and Step 3.
By Step 3, we have that
$$\overline{\O}\Big((\w_{y_0^*}^{\A'},\pmb\xi), \sigma\times \sigma\Big)=W_{y_0^*}^{\A'}\times \Sigma_k.$$
Since $y_0^*\in Y_0^*$, $\w_{y_0^*}^{\A'}$ is a distal point for $\sigma$ and it is product recurrent. Thus $(z, \w_{y_0^*}^{\A'})$ is a minimal point of $(Z\times W_{y_0^*}^{\A'}, S\times \sigma)$. By Step 3 again, we have that
$$\overline{\O}\Big((z, \w_{y_0^*}^{\A'},\pmb\xi),(S\times \sigma)\times \sigma\Big)=\overline{\O}\big((z,\w_{y_0^*}^{\A'}), S\times \sigma\big)\times \Sigma_k.$$
It follows that
$$(z, \w_{y_0^*}^{\A'}, \pmb\xi)\in \overline{\O}\Big((z, \w_{y_0^*}^{\A'}, \pmb\xi), (S\times \sigma)\times \sigma\Big),$$
and $(z, \w_{y_0^*}^{\A'}, \pmb\xi)$ is a piecewise syndetic recurrent point by Theorem \ref{prop-M} since minimal points are dense in $\overline{\O}\Big((z, \w_{y_0^*}^{\A'},\pmb\xi),(S\times \sigma)\times \sigma\Big).$

Recall that $y_0^*\in V^*$ and $T^{q_i(0)}x_0^*=x_0^*\in A_1^*\subseteq B^*_1$, $\forall i\in \{1,2,\ldots,r\}$. Hence $\pmb\xi_0=(1,1,\ldots,1)=1^{\otimes r}$.
Let
$$[\pmb{\xi}_{-l}^l]=\{\pmb\w\in \Sigma_k: \pmb\w_j=\pmb\xi_j: -l\le j\le l\}.$$
It is a neighbourhood of $\pmb\xi$ in $\Sigma_k$. Let
$$W^*=\{{\bf y}\in \big((Y^*)^r\big)^\Z: {\bf y}_j\in T^{q_1(j)}V^*\times T^{q_2(j)}V^*\times \cdots \times T^{q_r(j)}V^*, -l\le j\le l\}.$$
It is a neighbourhood of $\w_{y^*_0}^{\A'}$ in $\big((Y^*)^r\big)^\Z$.
So $ B(z,\ep)\times W^* \times [\pmb\xi_{-l}^l]$ is a neighbourhood of $(z, \w_{y_0^*}^{\A'}, \pmb\xi)$. Thus
$$N_2\triangleq \{n\in \Z:(S^nz, \sigma^n\w_{y_0^*}^\A, \sigma^n\pmb \xi)\in B(z,\ep)\times W^* \times [\pmb\xi_{-l}^l]\}$$
is piecewise syndetic.

Let $n\in N_2$.
By $\sigma^{n}\w_{y_0^*}^\A\in W^*$, we have $T^{q_i(n+j)}y_0^*\in T^{q_i(j)}V^*$ for all $i\in \{1,2,\ldots,r\}, j\in \{-l,-l+1,\ldots, l\}$. Hence for all $i\in \{1,2,\ldots, r\}, j\in \{-l, -l+1,\ldots, l\}$
\begin{equation}\label{e5}
  T^{q_i(n+j)}x_0^*\in (\pi^*)^{-1}(T^{q_i(j)}V^*)=\bigcup_{i=1}^k T^{q_i(j)}A_i^*.
\end{equation}
Since $\sigma^n\pmb\xi\in [\pmb\xi_{-l}^l]$, we have $\pmb\xi_{n+j}=\pmb\xi_j$ for all $j\in \{-l,-l +1,\ldots, l\}$. By \eqref{e2}, for all $i\in \{1,2,\ldots, r\}, j\in \{-l,-l+1,\ldots, l\}$
\begin{equation}\label{e6}
  T^{q_i(j)}x_0^*, T^{q_i(n+j)}x_0^*\in B_{\xi_j^{(i)}}^*.
\end{equation}
By \eqref{e5}, \eqref{e6} and the fact $|q_i(j)|\le L$, we have
$$ T^{q_i(j)}x_0^*, T^{q_i(n+j)}x_0^* \in T^{q_i(j)}A_{\xi_j^{(i)}}^*=B_{\xi_j^{(i)}}^*\cap (\pi^*)^{-1}(T^{q_i(j)}V^*)$$
for all $i\in \{1,2,\ldots, r\}, j\in \{-l,-l+1,\ldots, l\}$. So
$$ T^{q_i(j)}x_0', T^{q_i(n+j)}x_0' \in T^{q_i(j)}A_{\xi_j^{(i)}}=B_{\xi_j^{(i)}}\cap (\pi')^{-1}(T^{q_i(j)}V)$$
for all $i\in \{1,2,\ldots, r\}, j\in \{-l,-l+1,\ldots, l\}$.
Note that $B_i\cap (\pi')^{-1}(T^jV)=T^j A_i $ has diameter less than $\d$ for each $i\in \{1,2,\ldots, k\}$ and each $j\in \{-L,-L+1,\ldots, L\}$. In particular, we have that for all $i\in \{1,2,\ldots, r\}, j\in \{-l,-l+1,\ldots, l\}$,
$$\rho_{X'}(T^{q_i(j)}x_0', T^{q_i(n+j)}x_0')<\d.$$
By the definition of $\d$, we have
$$\rho_{X'}(T^{q_i(n+j)-q_i(j)}x_0', x_0')<\ep.$$
Combining with $S^nz\in B(z,\ep)$, we have
$$\rho_Z(S^nz,z)<\ep, \ \ \rho_{X'}(T^{q_i(n+j)-q_i(j)}x_0', x_0')<\ep,\ \forall  1\le i\le r, -l\le j\le l.$$
Since $\{p_{1},\ldots, p_d\}\subseteq \{q_i^{[j]}=q_i(n+j)-q_i(j): 1\le i\le r, -l\le j\le l\}$, we have
$$\rho_Z(S^n z,z)<\ep \ \text{and}\ \rho_{X'}(T^{p_1(n)}x_0', x_0')<\ep,\ \ldots,\ \rho_{X'}(T^{p_d(n)}x_0', x_0')<\ep.$$
So
$$N_2\subseteq N_{z,x_0',\ep},$$
and $N_{z,x_0',\ep}$ is piecewise syndetic.
The proof of Step 4. is complete.
\end{proof}

\medskip

\noindent {\bf Step 5.} {\em
Let $(Z, S)$ be a minimal t.d.s. and $z\in Z$. For each $\ep>0$ and $x_0\in X_\bullet$,
$$N_{z,x_0,\ep}=\{n\in \Z: \rho_Z(S^n z,z)<\ep \ \text{and}\ \rho_{X}(T^{p_1(n)}x_0, x_0)<\ep,\ \ldots,\ \rho_{X}(T^{p_d(n)}x_0, x_0)<\ep \}$$
is a piecewise syndetic subset.

In particular, there is some sequence $n_i\to\infty, i\to\infty$ such that
\begin{equation*}\label{}
  S^{n_i}z\to z,\ T^{p_1(n_i)}x_0\to x_0, \ T^{p_2(n_i)}x_0\to x_0,\ldots, \ T^{p_d(n_i)}x_0\to x_0, \ i\to\infty.
\end{equation*}
}

It follows from  Step 4, and the whole proof is completed. \hfill $\square$

\appendix

\section{Zero dimensional extensions}

In this section we give a proof of Theorem \ref{thm-0-dim-relative}.

\subsection{Inverse limits}

For a family of t.d.s. $\{(X_i, T_i)\}_{i\in I}$ where $I$ is a directed set, one can define the inverse system. We give definition for $I=\N$, general definition is similar.
For more information about inverse limits, see for example \cite[E.12]{Vr}.

Let $\{(X_i, T_i)\}_{i=1}^\infty$ be t.d.s. For any $i\ge j$ there are factor maps $\phi_{ij}: X_i\rightarrow X_j$ such that $\phi_{ii}=\id$  and
$$\phi_{jk}\circ \phi_{ij}=\phi_{ik}, \ \forall i\ge j\ge k.
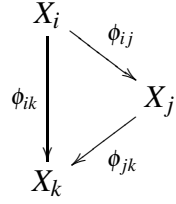
$$
$$\xymatrix@R=0.5cm{
  X_i \ar[dd]_{\phi_{ik}} \ar[dr]^{\phi_{ij}}             \\
                & X_j \ar[dl]^{\phi_{jk}}         \\
  X_k                 }
$$
Define a subset $X$ of $\prod_{i=1}^\infty X_i$ as
$$X=\left\{ x=(x_i)_{i=1}^\infty \in \prod_{i=1}^\infty X_i:
\phi_{ij}x_i=x_j, \ \forall i\ge j\right\},$$
with the relative topology induced on $X$ by the product topology on $\prod_{i=1}^\infty X_i$.

For $i\in\N$, let $\pi_i: X\rightarrow X_i, \pi_i(x)=x_i$ be projection. It is east to see that $\phi_{ij}\circ \pi_i=\pi_j, \forall i\ge j$.
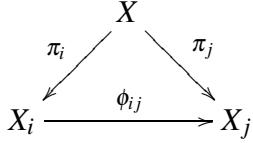
$$\xymatrix{
                & X\ar[dr]^{\pi_j}\ar[dl]_{\pi_i}             \\
 X_i \ar[rr]^{\phi_{ij}} & &     X_j        }
$$
Let $T: X\rightarrow X$ be defined as
$$T\big((x_i)_{i=1}^\infty\big)=(T_ix_i)_{i=1}^\infty,\ \forall (x_i)_{i=1}^\infty\in X.$$
It is clear that $$\pi_i: (X,T)\rightarrow (X_i, T_i)$$
is a factor map. We call that $(X,T)$ is the {\em inverse limit} of $\{(X_i, T_i)\}_{i=1}^\infty$, denoted by
$$(X,T)=\overleftarrow{\lim_{i\to\infty}} (X_i, T_i) \quad \text{or} \quad (X, T)={\rm inv\ }{\lim_{i\to\infty}} (X_i, T_i) .$$

\subsection{Admissible covers}

Let $X$ be a topological space.
If $\a,\b$ are covers of $X$, their {\em join} $\a\vee\b$ is the cover by all nonempty subset of the form $A\cap B$ where $A\in \a$ and $B\in \b$. Similarly we can define the join $\bigvee_{i=1}^n\a_i$ of any finite collection of covers of $X$.

A cover $\b$ is a {\em refinement} of a cover $\a$, written by $a\prec \b$  or $\b\succ \a$, if every member of $\b$ is a subset of a member of $\a$. It is clear that $\a\prec \a\vee \b$ for any covers $\a, \b$. Also if $\b$ is a subcover of $\a$, $\a\prec\b$.

Let $\partial\a=\bigcup_{A\in \a}\partial A$. It is clear that for any two covers $\a,\b$, $\partial(\a\vee\b)\subseteq \partial\a\cup\partial \b$.

\begin{de}
Let $X$ be a compact metric space. A cover $\a=\{A_1,\ldots, A_k\}$ of $X$ is called an {\em admissible cover} if each element is closed, $A_i\cap A_j=\partial A_i\cap \partial A_j$ if $i\neq j$ and $\partial\a=\bigcup_{i=1}^k\partial A_i$ having no interior.
\end{de}

\begin{lem}\label{lem1}
Let $X$ be a compact metric space and let $\a=\{A_1,\ldots, A_k\}$ be a closed cover of $X$. Then $\a$ is admissible if and only if ${\rm int}(A_i)\cap A_j=\emptyset$ for $i\neq j$ and the disjoint union $\bigcup_{i=1}^k{\rm int}(A_i)$ is dense in $X$.
\end{lem}

\begin{proof}
Let $\a=\{A_1,\ldots, A_k\}$ be an admissible cover. Since $A_i$ is closed, $A_i=\partial A_i\cup {\rm int}(A_i)$ for all $i\in \{1,2,\ldots, k\}$.
Then by the fact $A_i\cap A_j=\partial A_i\cap \partial A_j$ for $i\neq j$, we have
$${\rm int}(A_i)\cap A_j=\big(A_i\setminus \partial A_i\big)\cap A_j=\emptyset\ \text{for} \ i\neq j.$$ Since $X=\bigcup_{i=1}^kA_i=\bigcup_{i=1}^k \Big(\partial A_i\cup {\rm int}(A_i)\Big)=\big(\bigcup_{i=1}^k \partial A_i\big)\cup \big(\bigcup_{i=1}^k {\rm int}(A_i)\big)$ and $\bigcup_{i=1}^k\partial A_i$ has no interior, we have that $\bigcup_{i=1}^k{\rm int}(A_i)$ is dense in $X$.

Now we show the converse. Let $\a=\{A_1,\ldots, A_k\}$ be a closed cover.
For $i\neq j$, by the fact  ${\rm int}(A_i)\cap A_j=\emptyset$ and ${\rm int}(A_j)\cap A_i=\emptyset$, we have that
$$A_i\cap A_j=(\partial A_i\cup {\rm int}(A_i))\cap (\partial A_j\cup {\rm int}(A_j))=\partial A_i\cap \partial A_j.$$
Since $X=\big(\bigcup_{i=1}^k \partial A_i\big)\cup \big(\bigcup_{i=1}^k {\rm int}(A_i)\big)$ and  $\bigcup_{i=1}^k{\rm int}(A_i)$ is dense in $X$, we have that $\bigcup_{i=1}^k\partial A_i$ has no interior. That is, $\a$ is admissible.
\end{proof}

\begin{rem}\label{rem-lem1}
By the proof of Lemma \ref{lem1}, we also have that a closed finite cover $\a$ is admissible if and only if ${\rm int}(A)\cap B=\emptyset$ for $A\neq B\in \a$ and $\partial\a$ having no interior.
\end{rem}

\begin{lem}\label{lem2}
Let $X$ be a compact metric space. For any $\d>0$, there is some admissible cover $\a$ with ${\rm diam}(\a)=\min_{A\in \a} {\rm diam}(A)<\d$ and $A\neq \emptyset$ for all $A\in \a$.
\end{lem}

\begin{proof}
Let $\d>0$. Since $X$ is compact, we can find an open cover $\{B_1,\ldots, B_k\}$ by open balls of radius $\d/3$ and it has no proper subcovers. Let $A_1=\overline{B_1}$, and for $n\in \{2,3,\ldots, k\}$ let
$$A_n=\overline{B_n}\setminus \big(B_1\cup \cdots \cup B_{n-1}\big).$$
Then $\a=\{A_1,\ldots, A_k\}$ is a closed cover with ${\rm diam}(\a)<\d$.

First we show that $A_i\neq \emptyset$ for all $i\in \{1,2,\ldots, k\}$. Otherwise there is some $i\in \{1,2,\ldots, k\}$ such that $\overline{B_i}\setminus {(B_1\cup\cdots \cup B_{i-1})} =  A_i =\emptyset$. That is, $\overline {B_i}\subset {B_1\cup\cdots \cup B_{i-1}}$. Hence ${B_i}\subset B_1\cup \cdots \cup B_{i-1}$. It follows that $\{B_j\}_{j=1}^k\setminus \{B_i\}$ is a cover of $X$, which contradicts with the assumption that $\{B_j\}_{j=1}^k$ has no proper subcover.

Now we show that it is admissible. First we show that for $i<j$, $A_i\cap A_j=\partial A_i\cap \partial A_j$. Since ${\rm int}(A_j)=B_j\setminus \overline{(B_1\cup\cdots \cup B_{j-1})}$, we have
$$A_i\cap {\rm int} (A_j)=\Big(\overline{B_i}\setminus (B_1\cup\cdots \cup B_{i-1})\Big)\cap \Big(B_j\setminus \overline{(B_1\cup\cdots \cup B_{j-1})}\Big)=\emptyset$$
and hence $A_i\cap A_j= A_i\cap \partial A_j$. Note that
\begin{equation*}
  \begin{split}
     \partial A_j\cap {\rm int}(A_i)
   = & \partial \Big(\overline{B_j}\setminus (B_1\cup \cdots \cup B_{j-1})\Big)\cap \Big(B_i\setminus \overline{(B_1\cup\cdots \cup B_{i-1})}\Big)
    \\ \subseteq & B_i\setminus \big(B_1\cup\cdots \cup B_{j-1}\big)=\emptyset.
   \end{split}
\end{equation*}
It follows that $A_i\cap A_j= A_i\cap \partial A_j=\partial A_i\cap \partial A_j$. Also $$\partial\a=\bigcup_{i=1}^k\partial A_i=\bigcup_{i=1}^k\partial  \Big(\overline{B_i}\setminus (B_1\cup \cdots \cup B_{i-1})\Big)\subseteq \bigcup_{i=1}^k \bigcup_{j=1}^i\partial B_j =\bigcup_{i=1}^k \partial B_i,$$
which has no interior.
\end{proof}

\begin{lem}\label{lem3}
Let $X$ be a compact metric space. If $\a, \b$ are admissible covers of $X$, then $\a \vee \b$ is also admissible.

\end{lem}

\begin{proof}
Let $\a, \b$ be admissible covers of $X$. It is clear that $\a\vee \b$ is a closed cover of $X$. Let $A\cap B,A'\cap B'\in \a\vee \b$ where $A,A'\in \a, B,B'\in \b$. We show that if $A\cap B\neq A'\cap B'$, then
$${\rm int}(A\cap B)\cap (A'\cap B')=\emptyset.$$
Since $A\cap B\neq A'\cap B'$, one has that $A\neq A'$ or $B\neq B'$. By Lemma \ref{lem1}, ${\rm int} (A)\cap A'=\emptyset$ or ${\rm int}(B)\cap B'=\emptyset$. So
$${\rm int}(A\cap B)\cap (A'\cap B')\subseteq {\rm int}(A)\cap {\rm int}(B)\cap A'\cap B'=\big({\rm int} (A)\cap A'\big)\cap \big({\rm int} (B)\cap B'\big)=\emptyset.$$
Since $\partial (\a\vee\b)\subseteq \partial \a\cup \partial \b$, we have that $\partial (\a\vee\b)$ has no interior. Thus by Remark \ref{rem-lem1} $\a \vee \b$ is admissible.
\end{proof}

\subsection{Symbolic representation}
Let $(X,T)$ be a t.d.s.
Let $\a=\{A_1,\ldots, A_k\}$ be an admissible cover of $X$. Let
\begin{equation}\label{d1}
  D_\a=\bigcup_{n=-\infty}^\infty T^n\biggl(\bigcup_{i=1}^k\partial A_i\biggr)=\bigcup_{n=-\infty}^\infty T^n(\partial \a).
\end{equation}
Then $D_\a$ is a first category subset of $X$ as $\partial\a$ has no interior, so $X\setminus D_\a$ a residual subset of $X$ and it is invariant, i.e. $T(X\setminus D_\a)=X\setminus D_\a$. Define a map
\begin{equation}\label{d2}
\psi_\a: X\setminus D_\a\rightarrow \Sigma_{k}=\{1,2,\ldots,k\}^\Z \ \text{such that}\ x\mapsto (a_n)_{n\in \Z} \quad \text{if}\ T^nx\in A_{a_n}, \ \forall n\in \Z.
\end{equation}
We call $(a_n)_{n\in \Z}$ the {\em $\a$-name} of $x$.
Since $A_i\cap A_j=\partial A_i\cap \partial A_j$ for $i\neq j$, each point $x$ in $X\setminus D_\a$ has a unique $\a$-name and hence $\psi_\a: X\setminus D_\a\rightarrow \Sigma_{k}$
is well defined.

Let $X_\a=\overline{\psi_\a(X\setminus D_\a)}\subseteq \Sigma_k$. Then $\sigma X_\a=X_\a$ and $(X_\a,\sigma)$ is a subshift of $(\Sigma_k, \sigma)$. We call that the t.d.s. $(X_\a,\sigma)$ is the {\em symbolic representation} of $(X,\a)$.

\begin{rem}
\begin{enumerate}
  \item If $X_0\subseteq X\setminus D_\a$ is a dense subset of $X$, we also have that $X_\a=\overline{\psi_\a(X_0)}$.

  \item For each ${\bf a}=(a_n)_{n\in \Z}=\psi_a(x) \in \psi_a(X\setminus D_\a)$, we have that $\displaystyle x\in \bigcap_{n\in \Z}T^{-n}{\rm int}(A_{a_n})$. And for each ${\bf a}=(a_n)_{n\in \Z}\in X_\a$, we have $$\bigcap_{n\in \Z}T^{-n}\overline{{\rm int}(A_{a_n})}\neq \emptyset.$$
      Note if for some $A\in \a$, it has no interior, then $A$ will be useless in the symbolic representation of $(X,\a)$.

  \item Usually, $(X_\a,\sigma)$ is not a factor of $(X,T)$ nor an extension of $(X,T)$. But when the admissible cover $\a$ is also a generator, i.e. for every bisequence $\{A_{a_n}\}_{n\in \Z}$ of members of $\a$ the set $\displaystyle \bigcap_{n=-\infty}^\infty T^{-n}A_{a_n}$ contains at most one point of $X$, $(X_\a, \sigma)$ will be an extension of $(X,T)$. If in addition $(X,T)$ is a zero dimensional t.d.s. then $(X_\a,\sigma)$ is isomorphic to $(X,T)$.


\end{enumerate}
\end{rem}

\begin{lem}\label{lem4}
Let $(X,T)$ be a t.d.s. and let  $\a,\b$ be two admissible covers of $X$. If $\a\prec\b$, then $(X_\b,\sigma)$ is an extension of $(X_\a,\sigma)$.
\end{lem}

\begin{proof}
Let $\a=\{A_1,\ldots, A_k\}$ and $\b=\{B_1,\ldots, B_t\}$. Since $\a,\b$ are admissible and $\a\prec \b$, $t\ge k$ and
$$\partial \a=X\setminus \bigcup_{A\in \a}{\rm int}(A)\subseteq X\setminus \bigcup_{B\in \b}{\rm int}(B)=\partial \b.$$
Thus $D_\a\subseteq D_\b$, and $X\setminus D_\b\subseteq X\setminus D_\a$.

Define a map $\phi: \{1,2,\ldots, t\}\rightarrow \{1,2,\ldots, k\}$ such that $\phi(i)=j$ whenever $B_i\subseteq A_j$ (if $B_i$ is contained in more than one elements of $\a$, then pick smallest $j$ such that $B_i\subseteq A_j$). And $\phi$ induces a map from $\Sigma_t$ to $\Sigma_k$, we still denote it by $\phi$:
$$\phi: \Sigma_t\rightarrow \Sigma_k, \quad (a_n)_{n\in \Z}\mapsto \big(\phi(a_n)\big)_{n\in \Z}.$$
It is clear that $\rho_{\Sigma_k}\big(\phi({\bf a}),\phi({\bf b})\big)\le\rho_{\Sigma_t}\big({\bf a},{\bf b}\big)$ for all ${\bf a}, {\bf b}\in \Sigma_t$, where $\rho_{\Sigma_t},\rho_{\Sigma_k}$ are the metrics of $\Sigma_t$ and $\Sigma_k$ respectively. In particular, $\phi: \Sigma_t\rightarrow \Sigma_k$ is continuous.

Let $\psi_\a:X\setminus D_\a\rightarrow \Sigma_k$ and $\psi_\b:X\setminus D_\b\rightarrow \Sigma_t$  be defined in \eqref{d2}. Note that $X\setminus D_\b\subseteq X\setminus D_\a$ and both are dense in $X$.
We define a map $\phi_{\b\a}: \psi_\b(X\setminus D_\b)\rightarrow \psi_\a(X\setminus D_\b)$ by $\psi_\b(x)\mapsto \psi_\a(x)$ for all $x\in X\setminus D_\b$. Since for all $x\in X\setminus D_\b$, $x$ has a unique $\b$-name $\psi_\b(x)$ and a unique $\a$-name $\psi_\a(x)$, $\phi_{\b\a}: \psi_\b(X\setminus D_\b)\rightarrow \psi_\a(X\setminus D_\b)$ is well defined.
Note that $\phi_{\b\a}=\phi|_{\psi_\b(X\setminus D_\b)}$. It follows that $\phi_{\b\a}: \psi_\b(X\setminus D_\b)\rightarrow \psi_\a(X\setminus D_\b)$ is uniformly continuous and it can  be extended to a continuous map $\phi_{\b\a}: \overline{\psi_\b(X\setminus D_\b)}=X_\b\rightarrow X_\a$.
\begin{equation*}
  \xymatrix@R=0.5cm{
                &         X_\b \ar[dd]^{\phi_{\b\a}}     \\
  X\setminus D_\b \ar[ur]^{\psi_\b} \ar[dr]_{\psi_\a}                 \\
                &         X_\a                 }
\end{equation*}
Since $X\setminus D_\b$ is dense in $X$, $X_\a=\overline{ \psi_\a(X\setminus D_\b)}=\phi_{\b\a}(X_\b)$. It is easy to verify the following commutative diagram
\begin{equation*}
\xymatrix
{
\psi_\b(X\setminus D_\b) \ar[d]_{\phi_{\b\a}} & \psi_\b(X\setminus D_\b) \ar[l]_{\sigma}\ar[d]^{\phi_{\b\a}} \\
\psi_\a(X\setminus D_\b) & \psi_\a(X\setminus D_\b) \ar[l]^{\sigma} ,
}
\end{equation*}
and it follows that $\phi_{\b\a}:(X_\b,\sigma)\rightarrow (X_\a,\sigma)$ is a factor map.
\end{proof}

\begin{rem}\label{rem-lem4}
Let $(X,T)$ be a t.d.s. and let  $\a,\b$ and $\gamma$ be admissible covers of $X$. If $\a\prec\b\prec \gamma$, then by Lemma \ref{lem4} we have factor maps: $\phi_{\b\a}:X_\b\rightarrow X_\a$, $\phi_{\gamma\a}:X_\gamma\rightarrow X_\a$ and $\phi_{\gamma\b}:X_\gamma\rightarrow X_\b$. We claim that $\phi_{\gamma\a}=\phi_{\b\a}\circ \phi_{\gamma\b}$.

\medskip

Since $\a\prec\b\prec\gamma$, we have that $D_\gamma\subseteq D_\b\subseteq D_\a$. Thus $X\setminus D_\gamma\subseteq X\setminus D_\b\subseteq X\setminus D_\a$ and they all are dense in $X$. Restricted on $X\setminus D_\gamma$, we have maps:
$$ \psi_\gamma(X\setminus D_\gamma)\overset{\phi_{\gamma\b}}\longrightarrow \psi_\gamma(X\setminus D_\gamma) \overset{\phi_{\b\a}}\longrightarrow \psi_\a(X\setminus D_\gamma), \psi_\gamma(x)\mapsto \psi_\b(x)\mapsto\psi_\a(x), \forall x\in X\setminus D_\gamma.$$

\begin{equation*}
  \xymatrix@R=0.5cm{
                &        \psi_\gamma(X\setminus D_\gamma)\subseteq X_\gamma \ar[d]^{\phi_{\gamma\b}}
                 \\
              X\setminus D_\gamma \ar[r]^{\psi_\b}  \ar[ur]^{\psi_\gamma}\ar[dr]_{\psi_\a} &    \psi_\b(X\setminus D_\gamma)\subseteq     X_\b   \ar[d]^{\phi_{\b\a}}     \\
                & \psi_\a(X\setminus D_\gamma)\subseteq X_\a
                }
\end{equation*}
Note that on $\psi_\gamma(X\setminus D_\gamma)$, $\phi_{\gamma\a}=\phi_{\b\a}\circ \phi_{\gamma\b}$.
Since $X\setminus D_\gamma$ is dense in $X$, we have $X_\b=\overline{\psi_\b(X\setminus D_\gamma)}$
and $X_\a=\overline{\psi_\a(X\setminus D_\gamma)}$, and it follows
that the following diagram is commutative.
$$\xymatrix@R=0.5cm{
  X_\gamma \ar[dd]_{\phi_{\gamma\a}} \ar[dr]^{\phi_{\gamma\b}}             \\
                & X_\b \ar[dl]^{\phi_{\b\a}}         \\
  X_\a                }
$$
\end{rem}

\begin{lem}\label{lem5}
Let $(X,T)$ be a t.d.s. and let $\a$ be
an admissible cover of $X$. If $(X,T)$ is minimal, then $(X_\a,\sigma)$ is also a minimal t.d.s.
\end{lem}

\begin{proof}
Let $\a=\{A_1,\ldots, A_k\}$. We show that for each $x\in X\setminus D_\a$, $\psi_\a(x)$ is a minimal point of $(\Sigma_k,\sigma)$.
Let $\psi_\a(x)=(a_n)_{n\in \Z}$. To show that $\psi_\a(x)$ is a minimal point, we need to show that for each $N\in \N$, the word $(a_{-N}, a_{-N+1},\ldots, a_N)$ appears in $(a_n)_{n\in \Z}$ syndetically.

By definition of $\psi_\a$, we have that for each $N\in \N$, $x\in \bigcap_{j=-N}^NT^{-j}{\rm int}(A_{a_j})$. Since $\bigcap_{j=-N}^NT^{-j}{\rm int}(A_{a_j})$ is open and $x$ is minimal, $$N\left(x,\bigcap_{j=-N}^NT^{-j}{\rm int}(A_{a_j})\right )=\left\{n\in \Z: T^nx\in \bigcap_{j=-N}^NT^{-j}{\rm int}(A_{a_j})\right \}$$ is syndetic. Thus
$$\left \{n\in \Z: (a_{n-N}, a_{n-N+1},\ldots, a_{n+N})=(a_{-N}, a_{-N+1},\ldots, a_N) \right\}\supseteq N\left(x,\bigcap_{j=-N}^NT^{-j}{\rm int}(A_{a_j})\right )$$
is a syndetic subset.

Since $(X,T)$ is minimal, for each $x\in X\setminus D_\a$ the orbit $\{T^nx: n\in \Z\}\subseteq X\setminus D_\a$ is dense in $X$. Thus $X_\a=\overline{\psi_\a(\{T^nx: n\in \Z\})}$. But $\psi_\a(\{T^nx: n\in \Z\})=\{\sigma^n\psi_\a(x): n\in \Z\}$ is the orbit of $\psi_\a(x)$ in $\Sigma_k$. Hence $(X_\a, \sigma)$ is minimal.
\end{proof}

\subsection{Almost 1-1 symbolic extensions}

\begin{thm}\label{thm-0-dim}
Let $(X,T)$ be a t.d.s. Then it has an almost one to one zero dimensional extension, i.e. there is a zero dimensional t.d.s. $(X',T')$ such that it is an almost one to one extension of $(X,T)$.
If in addition $(X,T)$ is minimal, then $(X',T')$ is also minimal.
\end{thm}

\begin{proof}
By Lemma \ref{lem2} and Lemma \ref{lem3}, we can choose a sequence of closed covers $\{\a_n\}_{n=1}^\infty$ such that for each $n$, $\a_n$ is admissible with ${\rm diam }(\a_n)<\d_n$, and $\a_1\prec \a_2 \prec \cdots$, where $\{\d_n\}_{n=1}^\infty$ is a decreasing sequence with $\d_n\to 0, n\to\infty$. Let $(X_n,\sigma)$ is the symbolic representation of $(X,\a_n)$.

If $i=j$, then we set $\phi_{ii}=\id$. If $i>j$, then we define $\phi_{ij}: X_i\rightarrow X_j$ as $\phi_{ij}=\phi_{\a_i\a_j}$, where $\phi_{\a_i\a_j}$ is defined in Lemma \ref{lem4}.
Thus for any $i\ge j$ we have a factor map $\phi_{ij}: X_i\rightarrow X_j$ such that $\phi_{ii}=\id$  and by Remark \ref{rem-lem4}
$$\phi_{jk}\circ \phi_{ij}=\phi_{ik}, \ \forall i\ge j\ge k.
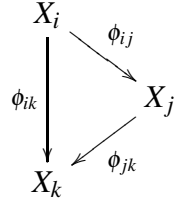
$$
$$\xymatrix@R=0.5cm{
  X_i \ar[dd]_{\phi_{ik}} \ar[dr]^{\phi_{ij}}             \\
                & X_j \ar[dl]^{\phi_{jk}}         \\
  X_k                 }
$$

Let $(X', T')={\rm inv\ }{\lim_{i\to\infty}} (X_i, \sigma).$ That is,
$$X' = \biggl\{ \w = (\w^i)_{i=1}^\infty \in \prod_{i=1}^\infty X_i:
\phi_{ij}(\w^i) = \w^j, \ \forall i\ge j\biggr\},$$
and $T'\Big((\w^i)_{i=1}^\infty\Big)=\Big((\sigma\w^i)_{i=1}^\infty\Big)$.

For $i\in \N$, let $D_i=D_{\a_i}$ and let $\psi_i=\psi_{\a_i}$, where $\psi_{\a_i}$ is defined in \eqref{d2}. Since $\a_1\prec \a_2 \prec\cdots$, one has that $D_1\subseteq D_2\subseteq \cdots$. Let $D_\infty=\bigcup_{i=1}^\infty D_{i}$. Then $D_\infty$ is still a first category subset of $X$, and $X\setminus D_\infty$ is a residual subset of $X$. So for all $i\in \N$, $X_i=\overline{\psi_i(X\setminus D_\infty)}$.
Let $$\varphi: X\setminus D_\infty\rightarrow X', \quad x\mapsto (\psi_i(x))_{i=1}^\infty.$$
And let
$$\widetilde{X}= \varphi(X\setminus D_\infty)=\biggl\{ \w = (\w^i)_{i=1}^\infty \in \prod_{i=1}^\infty X_i:
\w^i=\psi_i(x) , \ \forall i\in \N, \ \text{for some}\ x\in X\setminus D_\infty \biggr\}.$$
It is easy to see that $\widetilde{X}\subseteq X'$ and $\widetilde{X}$ is residual in $X'$.
Now we show that $\varphi: X\setminus D_\infty\rightarrow \widetilde{X}$ is one to one.
Let $x\neq y\in X\setminus D_\infty$. Then $\rho_X(x,y)>0$. Since $\{\d_n\}_{n=1}^\infty$ is a decreasing sequence with $\d_n\to 0, n\to\infty$, there is some $N\in\N$ such that $\d_n<\rho_X(x,y)/3$ for all $n\ge N$. As ${\rm diam} (\a_n)<\d_n$, $x,y$ will be in the different elements of $\a_n$ and we have $(\psi_n(x))_0\neq (\psi_n(y))_0$. In particular, $\psi_n(x)\neq \psi_n(y)$ and $\varphi(x)\neq\varphi(y)$.

Now we show that
$$\tau=\varphi^{-1}: \widetilde{X}\rightarrow X\setminus D_\infty,\quad (\psi_i(x))_{i=1}^\infty\mapsto x$$
is uniformly continuous. Therefore, $\tau$ can be extended to a continuous map $\tau: X'\rightarrow X$.

Let $\rho_i$ be the metric of $X_i$ for all $i\in \N$. That is, for ${\bf x}=(x_n)_{n\in \Z}, {\bf y}=(y_n)_{n\in \Z}\in X_i$,
$$\rho_i({\bf x}, {\bf y})=0\ \text{if}\ {\bf x}={\bf y};\ \  \rho_i({\bf x}, {\bf y})=\frac{1}{m+1}, \ \text{if}\ {\bf x}\not={\bf y}\ \text{and}\ m=\min\{|j|: x_j\neq y_j\}.$$
Let the metric $\rho_{X'}$ of $X'$ be defined as follows: for all $(\w^i)_{i=1}^\infty, (\widetilde{\w}^i)_{i=1}^\infty\in X'$
$$\rho_{X'}\Big((\w^i)_{i=1}^\infty, (\widetilde{\w}^i)_{i=1}^\infty\Big)=\sum_{i=1}^\infty\frac{\rho_i(\w^i,\widetilde{\w}^i)}{2^i}.$$

We need to show that for each $\ep>0$, there is some $\d>0$ such that whenever \linebreak $\rho_{X'}\Big((\psi_i(x))_{i=1}^\infty, (\psi_i(y))_{i=1}^\infty\Big)<\d$, then $\rho_X(x,y)<\ep$, where $\rho_X$ is the metric of $X$.
Since $\d_n\to 0, n\to\infty$, there is some $N\in \N$ such that $\d_N<\ep$. Let $\d<\frac{1}{2^N}$. Then
whenever $\rho_{X'}\Big((\psi_i(x))_{i=1}^\infty, (\psi_i(y))_{i=1}^\infty\Big)<\d$, we have that
$$\rho_N\big(\psi_N(x),\psi_N(y)\big)\le 2^N \rho_{X'}\Big((\psi_i(x))_{i=1}^\infty, (\psi_i(y))_{i=1}^\infty\Big)<2^N\d<1.$$
This implies $(\psi_N(x))_0=(\psi_N(y))_0$, which means that $x,y$ are in the same element of $\a_N$. Thus
$$\rho_X(x,y)\le {\rm diam}(\a_N)< \d_N< \ep.$$
Hence
$\tau: \widetilde{X}\rightarrow X\setminus D_\infty$
is uniformly continuous, and $\tau$ can be extended to a continuous map $\tau: X'\rightarrow X$.
For each $x\in X\setminus D_\infty$,
$$T\circ \tau\Big((\psi_i(x))_{i=1}^\infty\Big)=Tx=\tau\Big((\psi_i(Tx))_{i=1}^\infty\Big)=\tau\circ T'\Big((\psi_i(x))_{i=1}^\infty\Big).$$
Thus $T\circ \tau(x')=\tau\circ T'(x')$ for all $x'\in X'$, i.e. $\tau: (X',T')\rightarrow (X,T)$ is a factor map.
Since $\tau$ is one to one on $\widetilde{X}$, $(X',\sigma)$ is an almost one to one extension of $(X,T)$.

If in addition $(X,T)$ is minimal, then by Lemma \ref{lem5} $(X',T')$ is an inverse limit of minimal t.d.s. and it is also minimal.
\end{proof}

\subsection{Proof of Theorem \ref{thm-0-dim-relative}}

Now we generalize the result to a relative case. We need the following lemma.

\begin{lem}[{\cite[Theorem 4, Page 159]{Kura1}}]\label{lem6}
Let $X={\rm inv\ }{\lim_{i\to\infty}} X_i$, and $Y={\rm inv\ }{\lim_{i\to\infty}} Y_i$. If for each $i\in \N$ there is a map $h_i: X_i\rightarrow Y_i$ such that the following diagram is commutative: for all $i\le j$
\begin{equation*}
\xymatrix
{
X_i \ar[d]_{h_i} & X_j \ar[l]_{\phi_{ji}}\ar[d]^{h_j} \\
Y_i &  Y_j \ar[l]^{\phi_{ji}} ,
}
\end{equation*}
then we may define a map $h: X\rightarrow Y$ such that the following diagram is commutative: for all $i\in\N$
\begin{equation*}
\xymatrix
{
X_i \ar[d]_{h_i} & X \ar[l]_{\pi_i}\ar[d]^{h} \\
Y_i &  Y \ar[l]^{\pi_i} .
}
\end{equation*}
If each $h_i$ is continuous, then $h$ is also continuous.
\end{lem}

\begin{proof}[Proof of Theorem \ref{thm-0-dim-relative}]
First we have the following claim.

\medskip

\noindent {\bf Claim.} \ {\em  Let $\{\d_n\}_{n=1}^\infty$ be a decreasing sequence with $\d_n\to 0, n\to\infty$. There is a sequence of covers $\{\b_n\}_{n=1}^\infty$ of $Y$ and a sequence of covers $\{\a_n\}_{n=1}^\infty$ of $X$ such that
\begin{enumerate}
  \item for each $n\in \N$, $\b_n$ and $\a_n$ are admissible with ${\rm diam }(\b_n)<\d_n$ and ${\rm diam }(\a_n)<\d_n$;
  \item for each $n\in \N$, $\pi^{-1}(\b_n)\prec \a_n$;
  \item $\b_1\prec \b_2 \prec \cdots$ and $\a_1\prec \a_2 \prec \cdots$.
\end{enumerate}
}


\begin{proof}[Proof of Claim]
Let $\{\d_n\}_{n=1}^\infty$ be a strictly decreasing sequence with $\d_n>0$, and $\d_n\to 0, n\to\infty$.
By Lemma \ref{lem2} and Lemma \ref{lem3}, we can choose a sequence of closed covers $\{\b_n\}_{n=1}^\infty$ of $Y$ such that for each $n$, $\b_n$ is admissible with ${\rm diam }(\b_n)<\d_n$ and $\b_1\prec \b_2 \prec \cdots$. In addition, we may assume that for each $n\in \N$, $\bigcup_{B\in \b_n}\pi^{-1}(\partial B)$ has no interior. To see this, we need modify the proof of Lemma \ref{lem2}.

For each $y\in Y$ and $0< r<\d_n/3$, $\partial B(y,r)=\{z\in Y: \rho(y,z)=r\}$. Then there exists a $0< r_y<\d_n/3$ such that $\pi^{-1}(\partial B(y,r_y))$ has no interior. Otherwise $\{\pi^{-1}(\partial B(y,r))\}_{0<r<\d_n/3}$ is a family of uncountable disjoint closed subsets of $X$ and each element has non-trivial interior. This implies that $X$ is not a second countable space, which is a contradiction. Thus we have an open cover $\{B(y,r_y)\}_{y\in Y}$ of $Y$ such that $\pi^{-1}(\partial B(y,r_y))$ has no interior for each $y\in Y$ and $0<r_y<\d_n/3$. By compactness of $Y$ take a finite subcover $B_1,\ldots, B_k$ of $\{B(y,r_y)\}_{y\in Y}$.
Next we follow the proof of Lemma \ref{lem2}. Let $A_1=\overline{B_1}$, and for $n\ge 2$ let
$A_n=\overline{B_n}\setminus \big(B_1\cup \cdots \cup B_{n-1}\big).$
Like Lemma \ref{lem2}, we assume that ${\rm int}(A_i)\neq \emptyset$ for all $i\in \{1,2,\ldots, k\}$. Then $\b_n'=\{A_1,\ldots, A_k\}$ is a closed cover with ${\rm diam}(\b_n')<\d_n$.
It is admissible and $\bigcup_{j=1}^k\pi^{-1}(\partial A_k)\subseteq \bigcup_{j=1}^k\pi^{-1}(\partial B_k)$ has no interior. Next let $\b_n=\b_1'\vee\b_2'\vee\cdots \vee\b_n'$ for all $n\in \N$.
Then by Lemma \ref{lem3}, $\b_n$ is still admissible. Also we have ${\rm diam}(\b_n)\le {\rm diam}(\b_n')<\d_n$ and $\b_1\prec\b_2\prec\cdots$. And
$$\bigcup_{B\in \b_n}\pi^{-1}({\partial B})=\pi^{-1}\Big({\bigcup_{B\in \b_n}\partial B}\Big)\subseteq \pi^{-1}\Big(\bigcup_{i=1}^n\bigcup_{B'\in \b_i'}\partial \b'_i\Big)=\bigcup_{i=1}^n\bigcup_{B'\in \b_i'}\pi^{-1}\Big(\partial \b'_i\Big) $$
has no interior.

Hence we have that for each $n\in \N$, $\b_n$ is an admissible cover of $Y$ and $\bigcup_{B\in \b_n}\pi^{-1}(\partial B)$ has no interior. Let $\gamma_n=\{\pi^{-1}(B)\}_{B\in \b_n}$. We claim that $\gamma_n$ is admissible for all $n\in \N$. It is clear that $\gamma_n$ is a closed cover of $X$ as $\b_n$ is a closed cover.
Note that for each $B\in \b_n$,
$$\pi^{-1}(B)=\pi^{-1}({\rm int}(B))\cup\pi^{-1}(\partial (B)).$$
As $\partial B$ is a closed subset such that $\pi^{-1}(\partial B)$ has no interior points, we have that ${\rm int}(\pi^{-1}(B))=\pi^{-1}({\rm int}(B))$ and $\partial (\pi^{-1}(B))=\pi^{-1}(\partial B)$.
By Lemma \ref{lem1}, for any $B\neq B'\in \b_n$, ${\rm int}(B)\cap B'=\emptyset$.
Thus
$${\rm int}(\pi^{-1}(B))\cap \pi^{-1}(B')= \pi^{-1}({\rm int}(B))\cap \pi^{-1}(B')=\emptyset.$$
And $\bigcup_{B\in \b_n}\partial \pi^{-1}(B)=\bigcup_{B\in \b_n}\pi^{-1}(\partial B)$ has no interior.
By Remark \ref{rem-lem1}, $\gamma_n=\{\pi^{-1}(B)\}_{B\in \b_n}$ is an admissible cover of $X$.

By Lemma \ref{lem2}, choose an admissible cover $\gamma'_n$ of $X$ with ${\rm diam}(\gamma'_n)<\d_n$. Let $\a_1=\gamma_1\vee \gamma_1'$, and for $n\ge 2$, $\a_n=\gamma_n\vee \gamma'_n\vee \a_1\vee\cdots \vee \a_{n-1}$. Then by Lemma \ref{lem3}, $\a_n$ is an admissible cover with ${\rm diam}(\a_n)<\d_n$. Also by the construction of $\a_n$, we have that for each $n\in \N$, $\pi^{-1}(\b_n)\prec \a_n$, and $\a_1\prec\a_2\prec\cdots$. The proof of Claim is complete.
\end{proof}

Next we use Theorem \ref{thm-0-dim}. Let $(X_n,\sigma)$ is the symbolic representation of $(X,\a_n)$, and let $(Y_n,\sigma)$ is the symbolic representation of $(Y,\b_n)$. Let $(X', T')={\rm inv\ }{\lim_{i\to\infty}} (X_i, \sigma)$, and $(Y', S')={\rm inv\ }{\lim_{i\to\infty}} (Y_i, \sigma)$. Then by Theorem \ref{thm-0-dim}  $\tau: X'\rightarrow X$ and $\theta: Y'\rightarrow Y$ are almost one to one extensions.
By the Claim we have the following commutative diagram.
\begin{equation*}
	\xymatrix
	{X_1 \ar[d]_{\pi_1}             &
		X_2 \ar[l]_{\phi_{21}}
		\ar[d]_{\pi_2}
		 &
		X_3 \ar[l]_{\phi_{32}}
		\ar[d]_{\pi_3}   & \ar[l]
       \cdots
		& X_{n}\ar[l]
		\ar[d]_{\pi_{n}}
		 &
		X_{n+1}
		\ar[d]_{\pi_{n+1}}
		\ar[l]_{\phi_{n+1,n}} &  \ar[l]
		\cdots        & \ar[l]
		X'
		\ar[d]_{\pi'} \ar[r]^{\tau} & X  \ar[d]_{\pi}       \\
		Y_1                 &
		Y_2 \ar[l]^{\phi'_{21}}          &
		Y_3 \ar[l]^{\phi'_{32}} &  \ar[l]
		\cdots
		& Y_{n}  \ar[l]              &
		Y_{n+1}
		\ar[l]^{\phi'_{n+1,n}} &  \ar[l]
		\cdots         &\ar[l]
		Y' \ar[r]_{\theta} & Y
	}
	\end{equation*}

In the diagram, for each $n\in \N$, $\pi_n: X_n\rightarrow Y_n$ is induced by the fact that $\pi^{-1}(\b_n)\prec \a_n$. By Lemma \ref{lem6}, the family $\{\pi_n\}_n$ will define the continuous map
$$\pi': (X', T')={\rm inv\ }{\lim_{i\to\infty}} (X_i, \sigma)\rightarrow{\rm inv\ }{\lim_{i\to\infty}} (Y_i, \sigma)=(Y', S').$$
Since each $\pi_n$ is a factor map, it is easy to verify that $\pi'$ is also a factor map.

If in addition $(X,T)$ and $(T,T)$ are minimal, by Theorem \ref{thm-0-dim} $(X',T')$ and $(Y',S')$ are minimal.
The proof is complete.
\end{proof}


\end{document}